%% file: TAC_FM_2019.tex
\documentclass[journal,a4paper]{IEEEtran}

\usepackage{amsmath}
\usepackage{amssymb} 
\usepackage{amsthm}
\usepackage{tabularx}
\usepackage{dsfont}
\usepackage[mathscr]{eucal}
\usepackage{algorithm}
\usepackage{algpseudocode}
\usepackage{cite}

\makeatletter
\let\NAT@parse\undefined
\makeatother
\usepackage{hyperref}

\usepackage{enumerate}

\usepackage{graphicx}      
\graphicspath{{./figs/}}

\usepackage[dvipsnames]{xcolor} %

\newcommand\traspo{{\mathpalette\raiseT\intercal}}
\newcommand\raiseT[2]{\raisebox{0.25ex}{$#1#2$}}
\newcommand{\trasp}{\ensuremath{^\traspo}}

\DeclareMathOperator*{\argmax}{arg\,max}
\DeclareMathOperator*{\argmin}{arg\,min}

\newcommand{\setN}{\ensuremath{\mathcal{N}}}

\newcommand{\setX}{\ensuremath{\mathcal{X}}}

\newcommand{\setC}{\ensuremath{\mathcal{C}}}

\newcommand{\setD}{\ensuremath{\mathcal{D}}}

\newcommand{\setY}{\ensuremath{\mathcal{Y}}}

\newcommand{\xcirc}{\ensuremath{x^{\circ}}}
\newcommand{\xopt}{\ensuremath{x^{\ast}}}
\newcommand{\yopt}{\ensuremath{y^{\ast}}}

\newcommand{\zopt}{\ensuremath{z^{\ast}}}
\newcommand{\gammaopt}{\ensuremath{\gamma^{\ast}}}

\newcommand{\xmenoi}{\ensuremath{x_{-i}}}

\newcommand{\game}{\ensuremath{\mathcal{G}}}

\newcommand{\gameaug}{\ensuremath{\widehat{\mathcal{G}}}}

\newtheorem{thm}{Theorem}
\newtheorem{cor}[thm]{Corollary}
\newtheorem{lem}[thm]{Lemma}
\newtheorem{prop}[thm]{Proposition}
\newtheorem{defn}[thm]{Definition}
\newtheorem{assum}[thm]{Assumption}

\newtheorem{rem}[thm]{Remark}

\begin{document}
\title{Probably Approximately Correct\\Nash Equilibrium Learning}
\author{Filiberto~Fele~and~Kostas~Margellos 
	\thanks{Research was supported by the UK Engineering and Physical Sciences Research Council (EPSRC) under grant agreement EP/P03277X/1.
	}
	\thanks{The authors are with the Department of Engineering Science,
		University of Oxford, OX1 3PJ, UK
		{\tt\small \{filiberto.fele, kostas.margellos\}@eng.ox.ac.uk}}%
}

\maketitle

\begin{abstract}
We consider a multi-agent noncooperative game with agents' objective functions being affected by uncertainty. Following a data driven paradigm, we represent uncertainty by means of scenarios and seek a robust Nash equilibrium solution. 
We treat the Nash equilibrium computation problem within the realm of probably approximately correct (PAC) learning.
Building upon recent developments in scenario-based optimization, we accompany the computed Nash equilibrium with \emph{a priori} and \emph{a posteriori} probabilistic robustness certificates, providing confidence that the computed equilibrium remains unaffected (in probabilistic terms) when a new uncertainty realization is encountered. 
For a wide class of games, we also show that the computation of the so called compression set --- a key concept in scenario-based optimization --- can be directly obtained as a byproduct of the proposed solution methodology.
Finally, we illustrate how to overcome differentiability issues, arising due to the introduction of scenarios, and compute a Nash equilibrium solution in a decentralized manner. 
We demonstrate the efficacy of the proposed approach on an electric vehicle charging control problem.
\end{abstract}


\input{intro}

\input{results}
\input{solutionVI}

\input{example}

\addtolength{\textheight}{-19.5cm}

\bibliographystyle{IEEEtran}
\bibliography{biblioSBDR.bib}

\end{document}

%% file: intro.tex
\section{Introduction}
\label{sec:introduction}
\IEEEPARstart{G}{ame} theory has attracted significant attention in the control systems community~\cite{LAMNABHILAGARRIGUEEtAl2017CPS}, 
and has found numerous applications ranging from smart grid \cite{Tony2019,DePaolaEtAl2018CDC,AtzeniEtAl2014TSP} and electricity markets \cite{Oren2004,Oren2008}, to communication networks \cite{ScutariEtAl2014} and regulatory compliance \cite{Krawczyk2007,Ohlin2012}.
The concept of Nash equilibrium (NE) is central in this context, as it defines no-regret strategies for noncooperative, selfish agents~~\cite{Nash1951,BasarOlsder1998BOOK}.
As a result, NE has been a popular solution for multi-agent distributed and decentralized control architectures, investigating also the connections with social welfare; see, e.g., \cite{FacchineiKanzow2007,Grammatico2017TAC, DeoriEtAl2018AUT, PaccagnanEtAl2018TAC, PaccagnanEtAl2018CSL,FeleEtAl2018ARC} and references therein. 

Uncertainty has been widely addressed in Nash games, by adopting stochastic or worst-case approaches. In the first case, both chance-constrained (risk-averse) \cite{Couchman20051283,SINGH2016640,PangEtAl2017} or expected payoff criteria~\cite{Ravat20111168,KoshalEtAl2013TAC,XuZhang2013,YuEtAl2017Learning,LeiShanbhag2017CDC} have been considered. For tractability, these methods typically involve assumptions on the underlying probability distribution of uncertainty realization. In the second case, results build upon robust control theory~\cite{AghassiBertsimas2006,Nishimura09robustnash,ZazoEtAl2017}; however, these rely on certain assumptions on the geometry of the uncertainty set; see~\cite{HuFukushima2013,PangEtAl2017}. 

In this paper we consider a multi-agent NE seeking problem with uncertainty affecting agents' objective functions. We depart from existing paradigms and follow a data driven methodology, where we represent uncertainty by a finite set of scenarios that could either be extracted from historical data, or by means of some prediction model (e.g., regression, Markov chains, neural networks) \cite{ConejoEtAl2010BOOK}. Adopting a data driven methodology poses a main challenge: 
NE are inherently random as they depend on the observed scenarios. Therefore, our objective is to investigate the sensitivity of the resulting NE to the uncertainty, in a probabilistic sense.
More specifically, our contributions can be outlined as follows:

1) We treat the NE computation problem in a probably approximately correct (PAC) learning framework \cite{FloydWarmuth1995,AlamoEtAl2009TAC,MargellosEtAl2015TAC}, and employ the so called scenario approach \cite{CalafioreCampi2006TAC}. Building on \cite{CampiEtAl2018TAC} we first provide an \emph{a posteriori} certificate on the probability that a NE remains unaltered upon a new realization of the uncertainty.
We then rely on \cite{CampiGaratti2008SIAM} and provide an \emph{a priori} probabilistic certificate on the equilibrium sensitivity, under an additional non-degeneracy assumption (see Section \ref{sec_thmresult} for a definition). The obtained results are distribution-free, and as such the underlying probability distribution of the uncertainty could be unknown and the only requirement is the availability of samples. 

Blending the scenario approach with game theory has only recently appeared in the literature \cite{FeleMargellos2019,PaccagnanCampi2019}. The validity of the probabilistic statements presented in this paper extends to games admitting multiple NE. Moreover, all our \emph{a posteriori} statements allow for degenerate problem instances; the latter circumvents the need of verifying the non-degeneracy assumption which, unlike convex optimization programs, is often not satisfied in games.

2) Under the additional assumption that the game under consideration admits a unique NE, or for aggregative games with multiple equilibria but a unique aggregate solution, we show that a \emph{compression set} (a key concept in learning and generalization --- see Section~ \ref{sec_thmresult} for a definition) can be directly computed by inspection of the solution returned by the proposed algorithm. This feature has significant computational advantages as it prevents the use of a greedy mechanism (see, e.g., \cite{CampiEtAl2018TAC}), which would require running up to numerical convergence multiple times (possibly as many as the number of samples) a NE seeking algorithm (see Section~\ref{sec_apost_eval}).

3) We provide a constructive proof of the existence of a single-valued mapping from the set of observed scenarios to a NE of the robust game, where the latter possibly admits multiple equilibria (and multiple maximisers). More specifically, we build an iterative algorithm for decentralized NE computation.
To circumvent nondifferentiability issues and incorporate an equilibrium selection mechanism we bridge the results in~\cite{FacchineiEtAl2014} and \cite{ScutariEtAl2014} that involve resorting to an augmented $\min-\max$ game. The proposed scheme enjoys the same convergence properties as state-of-the-art decentralized algorithms for monotone games \cite{ScutariEtAl2014} (see Section~\ref{sec_extgame}).

Note that the results presented in this paper do not contemplate constraints coupling agents' strategies. The latter give rise to \emph{generalized} NE problems; we refer the reader to \cite{FacchineiKanzow2007,PaccagnanEtAl2018TAC,Noor1996,Kannan_etal_2013} for details.


In Section~\ref{sec_thmresult} we introduce the scenario-based Nash game, pose the main problem, and present the main results of the paper. Section~\ref{sec_extgame} provides a decentralized construction of a solution algorithm for the game under study, while Section~\ref{sec_thmproof} contains the proof of the main results.
In Section~\ref{sec_apost_eval} we provide for a wide class of games a computationally efficient methodology to determine an upper bound to the cardinality of the compression set. Section~\ref{sec_example} provides an electric-vehicle charging control case study, while Section~ \ref{sec_conclusion} concludes the paper and provides some directions for future work.

%% file: results.tex
\section{Scenario based multi-agent game}
\label{sec_thmresult}

\subsection{Gaming set-up}
\label{sec_setup}

Let the set $\setN = \{1,\ldots,N\}$ designate a finite population of agents.
The decision vector, henceforth referred to as strategy, of each agent $i\in\setN$ is denoted by $x_i\in\mathbb{R}^{n}$ and satisfies individual constraints encoded by the set $\setX_i\subset\mathbb{R}^n$. 
We denote by $x = (x_i)_{i\in\setN}\in\setX\subset\mathbb{R}^{nN}$ the collection of all agents' strategies, where $\setX = \setX_1 \times \cdots \times \setX_{N}$. For any agent $i\in\setN$, $\xmenoi = (x_j)_{j\in\setN\setminus\{i\}}$ denotes the collection of strategies of all other agents.

Let $\theta$ be an uncertain vector taking values over some set $\Theta$, endowed with a $\sigma$-algebra $\mathscr{Q}$, and let $\mathds{P}$ denote the probability measure defined over $\mathscr{Q}$. 
For all subsequent derivations fix any $M \in \mathbb{N}$, and let $(\theta_1,\ldots,\theta_M) \in \Theta^M$ be a finite collection of independent and identically distributed (i.i.d.) scenarios/realizations of the uncertain vector $\theta$, that we will henceforth designate as $M$-multisample.
For given strategies of the remaining agents $\xmenoi$, each agent $i\in\setN$ aims at minimizing with respect to $x_i$ the function
\begin{equation}\label{eq_obj}
J_i(x_i,\xmenoi) = f_i(x_i,\xmenoi) + \max\limits_{m \in\{1,\ldots,M\}} g(x_i,\xmenoi,\theta_m),
\end{equation}
where $f_i: \mathbb{R}^{nN} \to \mathbb{R}$ expresses a deterministic objective, different for each agent $i$ but still dependent on the strategies of all agents, while $g: \mathbb{R}^{nN} \times \Theta \to \mathbb{R}$ encodes a common component in the agents' objective function that depends on the uncertain vector.
Agents are interested in minimizing their local objective $f_i$ and the worst-case (maximum) value $g$ can take among a finite set of scenarios.
The electric vehicle charging control problem of Section~\ref{sec_example} provides a natural interpretation of such a set-up, where electric vehicles are selfish entities each one with a possibly different utility function $f_i$; however, they could be participating in the same aggregation plan or belonging to a centrally managed fleet, thus giving rise to a common $g$.

We consider a noncooperative game among the $N$ agents, described by the tuple $\game = \langle \setN,(\setX_i)_{i\in\setN}, (J_i)_{i\in\setN},\{ \theta_m \}_{m=1}^M \rangle$, where $\setN$ is the set of agents/players, $\setX_i$, $J_i$ are respectively the strategy set and the cost function for each agent $i\in\setN$, and $\{\theta_1, \ldots,\theta_M\}$ is a finite collection of samples.
We consider the following solution concept for $\game$:
\begin{defn}[Nash equilibrium]\label{defn_NE}
	Let $\Omega \subseteq \setX$ denote the set of Nash equilibria of $\game$, defined as
	\begin{multline}\label{eq_defNE}
	\Omega = \big\{ \xopt = (\xopt_i)_{i\in\setN} \in \setX \colon  \\
 \xopt_i \in \argmin_{x_i\in\setX_i} J_i(x_i,\xopt_{-i}), \, \forall i\in\setN\big\}.
	\end{multline}
\end{defn}

We impose the following standing assumptions:
\begin{assum}\label{assum_eqNE_VIsol} 
	(i) For any $\theta\in\Theta$, and any $x_{-i} \in \setX_{-i} = \Pi_{j \neq i \in \setN} \setX_j$, $f_i(\cdot,x_{-i}) + g(\cdot,x_{-i},\theta)$ is convex and continuous differentiable, while the local constraint set $\setX_i$ is nonempty, compact and convex for all $i\in\setN$. \\
	(ii) For any $\theta\in\Theta$, and for all $i\in\setN$, the functions $g$ and $f_i$ are twice differentiable on an open convex set containing $\setX$. \\
	(iii) The pseudo-gradient $(\nabla_{x_i} f_i(x))_{i\in\setN}$ is monotone with constant $\chi^f \in \mathbb{R}$, while $\nabla_{x} g(x,\theta)$ is monotone with constant $\chi^g \in \mathbb{R}$ for any fixed $\theta$, i.e., for any $u, v \in \mathbb{R}^{nN}$, and $\theta \in \Theta$,
	\begin{align}\label{eq_assum2}
	& (u-v)\trasp ((\nabla_{u_i} f_i(u))_{i\in\setN} - (\nabla_{v_i} f_i(v))_{i\in\setN}) \geq \chi^f \|u-v\|^2, \nonumber \\
	& (u-v)\trasp (\nabla_{u} g(u,\theta) - \nabla_{v} g(v,\theta))  \geq \chi^g \|u-v\|^2,
	\end{align}
	and $\chi^f + \chi^g \geq 0$. 
\end{assum}

We wish to emphasize that convexity of $f_i(\cdot,x_{-i}) + g(\cdot,x_{-i},\theta)$ (for any fixed $x_{-i}$ and $\theta$) does not require $f_i(\cdot,x_{-i})$ and $g(\cdot,x_{-i},\theta)$ to be both convex; indeed, Assumption \ref{assum_eqNE_VIsol} allows either function to be \emph{weakly} convex.
Similarly, it is not required that both $\chi^f, \chi^g$ are non-negative, but only $\chi^f + \chi^g \geq 0$. 
Note that a sufficient (but stronger) condition for the monotonicity requirements to be satisfied is for
$f_i + g$ to be jointly convex with respect to $x$. 

\subsection{Problem statement}
\label{sec_probstate}

As every NE $\xopt\in\Omega$ is a random vector due to its dependency on the $M$-multisample, a question that naturally arises is how sensitive a NE is against a new realization of the uncertainty.
More formally, let $\xopt\in\Omega$ be a NE of the game with $M$ samples. Consider a new realization $\theta \in \Theta$, and
let $\game^+ = \langle \setN,(\setX_i)_{i\in\setN}, (J_i)_{i\in\setN},\{ \theta_m \}_{m=1}^M \cup \{\theta\} \rangle$ be the game defined over the $M+1$ scenarios $\{\theta_1, \ldots,\theta_M,\theta\}$; denote by $\Omega^+$ the set of the associated NE. 
Then, for all $\xopt\in\Omega$, let
\begin{equation}
	V(\xopt) = \mathds{P}\{\theta\in\Theta\colon\, \xopt\notin\Omega^+ \}
\end{equation}
denote the probability that a NE of $\game$ does not ``remain'' a NE of $\game^+$, i.e., of the game characterized by the extraction of an additional sample. Note that $V(\xopt)$ is in turn a random variable, as its argument depends on the multisample $\{\theta_1, \ldots,\theta_M\}$.
To provide a rigorous answer to the above question we will study the generalization properties of $\xopt$ within a probably approximately correct (PAC) learning framework.
With a given confidence/probability with respect to the product measure $\mathds{P}^{M}$ (as the samples are extracted in an i.i.d.~fashion), \emph{we aim at quantifying} $V(\xopt)$.

To achieve such a characterization we provide some basic definitions. Let $\Phi\colon \Theta^M \to \Omega$ be a \emph{single-valued} mapping from the set of $M$-multisamples to the set of equilibria of $\game$.
\begin{rem} \label{rem:family_M}
The game $\game$, the set of NE $\Omega$, the mapping $\Phi$ (as well as of other associated quantities introduced in the sequel) depend on $M$ via the $M$-multisample employed. Therefore, they are parameterized by $M$, giving rise to a family of games, NE sets and mappings. To ease notation we do not show this dependency explicitly. Also, the dimension of the domain of $\Phi$ is to be intended in accordance with $M$.
\end{rem}
\begin{defn}[Support sample~\cite{CampiGaratti2008SIAM}]\label{defn_suppsample} 
	Fix any i.i.d.~$M$-multisample $(\theta_1,\ldots,\theta_M)\in\Theta^M$, and let $\xopt = \Phi(\theta_1,\ldots,\theta_M)$ be a NE of $\game$. Let $\xcirc = \Phi(\theta_1,\ldots,\theta_{s-1},\theta_{s+1},\ldots,\theta_M)$ be the solution obtained by discarding the sample $\theta_s$. We call the latter a \emph{support sample} if $\xcirc \neq \xopt$.
\end{defn}

\begin{defn}[Compression set --- adapted from~\cite{CampiEtAl2018TAC}]\label{defn_CS} 
	Fix any i.i.d. $M$-multisample $(\theta_1,\ldots,\theta_M)\in\Theta^M$, and let $\xopt = \Phi(\theta_1,\ldots,\theta_M)$ be a NE of $\game$. Consider any subset $\setC\subseteq \{\theta_1,\ldots,\theta_M\}$ and let $\xcirc = \Phi(\setC)$. We call $\setC$ a compression set if $\xcirc = \xopt$.
\end{defn}
The notion of compression set has appeared in the literature under different names; its properties are studied in full detail in \cite{CampiEtAl2018TAC}, where it is designated as support subsample. Here we adopt the term compression set as in \cite{FloydWarmuth1995,MargellosEtAl2015TAC} to avoid confusion with  Definition~\ref{defn_suppsample}.

Let $\mathfrak{C}(\theta_1,\ldots,\theta_M)$ be the collection of all compression sets associated with the $M$-multisample $\{\theta_1,\ldots,\theta_M \}$. We refer to the cardinality $|\setC|$ of some compression set $\setC\in\mathfrak{C}(\theta_1,\ldots,\theta_M)$ as the \emph{compression cardinality} $d^*$ (we do not  make explicit the dependence on the specific $\setC$ to ease notation, as the results below hold for any compression set in $\mathfrak{C}$). 
Note that $\mathfrak{C}$ --- hence also $d^*$ --- is itself a random variable as it depends on the $M$-multisample.

\begin{defn}[Non-degeneracy --- adapted from~\cite{CampiGaratti2018MATH}]\label{def_nondeg}
For any $M \in \mathbb{N}$, with $\mathds{P}^M$-probability equal to $1$, the NE $\xopt = \Phi(\theta_1,\ldots,\theta_M)$ coincides with the NE returned by $\Phi$ when the latter takes as argument \emph{only} the support samples. 
The corresponding game is then said to be non-degenerate; otherwise it is called degenerate.
\end{defn}
It follows that for non-degenerate problems the support samples form a compression set with $\mathds{P}^M$-probability 1. For degenerate problems the notions in Definitions \ref{defn_suppsample} and \ref{defn_CS} do not necessarily coincide; in particular, the support samples form a strict subset of any compression set in $\mathfrak{C}$. For a detailed discussion on degeneracy in scenario-based contexts, we refer the reader to \cite{Calafiore2010SIAM,CampiGaratti2018MATH}. 

\subsection{Main results}
\label{sec_thmstate}

We first show that a single-valued mapping $\Phi\colon \Theta^M \to \Omega$ from the set of $M$-multisamples to the set of NE of the game $\game$ indeed exists.
\begin{prop} \label{prop_decen}
Under Assumption~\ref{assum_eqNE_VIsol} there exists a single-valued decentralized mapping $\Phi\colon \Theta^M \to \Omega$.
\end{prop}
The mapping can be computed in a decentralized manner, thus fitting the inherent structure of the game. Its construction, and hence the proof of Proposition \ref{prop_decen}, is provided in Section~\ref{sec_extgame}.

\subsubsection{\emph{A posteriori} certificate}
We provide an \emph{a posteriori} quantification of an upper bound for $V(\xopt)$. This is summarized in the following theorem.
\begin{thm}\label{thm_apost}
Consider Assumption~\ref{assum_eqNE_VIsol}.
Fix $\beta\in(0,1)$ and let $\varepsilon\colon\{0,\ldots,M\}\to [0,1]$ be a function satisfying
		\begin{equation}\label{eq_eps}
		\varepsilon(M) = 1,\text{ and }\sum_{k=0}^{M-1}\binom{M}{k} (1-\varepsilon(k))^{M-k} = \beta.
		\end{equation}
		Let $\xopt = \Phi(\theta_1,\ldots,\theta_M)$, where $(\theta_1,\ldots,\theta_M)\in\Theta^M$ is an i.i.d.~sample from $\Theta$. We then have that
		\begin{equation}\label{eq_resthm}
		\mathds{P}^M\{(\theta_1,\ldots,\theta_M) \in \Theta^M:\, V(\xopt) \leq \varepsilon(d^*) \} \geq 1-\beta,
		\end{equation}
		where $d^*\leq M$ is the cardinality of any compression set of $\{\theta_1,\ldots,\theta_M\}$.
\end{thm}
\vspace{0.2cm}

Theorem~\ref{thm_apost} shows that --- with confidence at least $1-\beta$ --- the probability that a NE $\xopt$ of $\game$ does not ``remain'' an equilibrium of $\game^+$ (i.e., when an additional sample $\theta\in\Theta$ is considered) is at most $\varepsilon(d^*)$. Note that \eqref{eq_resthm} captures the generalization properties of $\xopt$, where $1-\beta$ accounts for the `probably' and $\varepsilon(d^*)$ for the `approximately correct' term used within a PAC learning framework.
The value of $\varepsilon(\cdot)$ is defined in accordance to \cite{CampiEtAl2018TAC} and depends on the \emph{observed} compression cardinality $d^*$, which in turn depends on the random multiextraction $\{\theta_1,\ldots,\theta_M\}$ thus giving rise to the \emph{a posteriori} nature of the result. As a consequence, the level of conservatism of the obtained certificate depends on $d^*$; the smaller the cardinality of the computed compression set, the tighter the bound; see Section~\ref{sec_apost_eval} for a detailed elaboration on the computation of $d^*$. 
The proof of Theorem~\ref{thm_apost} is provided in Section~\ref{sec_thmproof_apost}.

In the case of a non-degenerate game (see Definition \ref{def_nondeg}), the bound could be significantly improved by means of the \emph{wait-and-judge} analysis of~\cite{CampiGaratti2018MATH}: specifically, by Theorem 2 in \cite{CampiGaratti2018MATH}, we can replace the expression for $\varepsilon(\cdot)$ in \eqref{eq_eps} with $\varepsilon(k) = 1-t(k)$, where $t(k)$ is the unique solution in $(0,1)$ of
\begin{equation}
\frac{\beta}{M+1} \sum_{m=k}^M {m \choose k} t^{m-k} - {M \choose k} t^{M-k} = 0. \label{eq:eps_tight}
\end{equation}
However, note that non-degeneracy is a condition in general difficult to verify even in convex optimization settings, a challenge that becomes more prominent in games.


\subsubsection{\emph{A priori} certificate}
We now provide an \emph{a priori} quantification of an upper-bound of $V(\xopt)$. This is summarized in the following theorem.

\begin{thm}\label{thm_apriori}
Consider Assumption~\ref{assum_eqNE_VIsol}, and further assume that the game is non-degenerate according to Definition \ref{def_nondeg}.
Fix $\beta\in(0,1)$ and consider $\varepsilon\colon\{0,\ldots,M\}\to [0,1]$ be a function satisfying \eqref{eq_eps}.	
Let $\xopt = \Phi(\theta_1,\ldots,\theta_M)$, where $(\theta_1,\ldots,\theta_M)\in\Theta^M$ is an i.i.d.~multisample. We then have that
		\begin{equation}\label{eq_resthm_apriori}
		\mathds{P}^M\{(\theta_1,\ldots,\theta_M) \in \Theta^M:\, V(\xopt) \leq \varepsilon((n+1)N) \} \geq 1-\beta.
		\end{equation}
\end{thm}
\vspace{0.2cm}

The proof of Theorem \ref{thm_apriori} is provided in Section~\ref{sec_thmproof_new}.
Although similar in form to Theorem \ref{thm_apost}, the bound on $V(\xopt)$ provided by Theorem \ref{thm_apriori} additionally relies on the developments in \cite{CampiGaratti2008SIAM,Calafiore2010SIAM}: these results are independent of the given multisample and linked instead to the problem structure. In this way, $\varepsilon(\cdot)$ is evaluated on the sample-independent quantity $(n+1)N$, expressing the dimension $nN$ of the agents' decision space plus $N$ additional variables, explained by the epigraphic reformulation introduced in the proof of Theorem \ref{thm_apriori}. If we further assume that for all $i\in\setN$, for every fixed $x_{-i}\in\setX_{-i}$ and $\theta\in\Theta$, both $f_i(\cdot,x_{-i})$ and $g(\cdot,x_{-i},\theta)$ are convex, we would only need one epigraphic variable, hence the argument of $\varepsilon(\cdot)$ could be replaced by $nN+1$ (see Section~\ref{sec_thmproof_new}).

Since we strengthen here the assumptions of Theorem \ref{thm_apriori} by imposing a non-degeneracy condition (see Definition \ref{def_nondeg}), \eqref{eq_eps} could be directly replaced by the tighter expression \eqref{eq:eps_tight}.
We wish to emphasize that, even if the non-degeneracy assumption holds, it may still be preferable to calculate the cardinality $d^*$ in an \emph{a posteriori} fashion, as in certain problems the latter might be significantly lower compared to $(n+1)N$. This is also the case in the electric vehicle charging control problem of Section \ref{sec_example}.

\begin{cor} \label{cor:cost_deter}
Let $\xopt = \Phi(\theta_1,\ldots,\theta_M)$ and define
\begin{displaymath}
	V_c(\xopt) = \mathds{P}\{\theta\in\Theta\colon g(\xopt,\theta)>\max_{m\in\{1,\ldots,M\}} g(\xopt,\theta_m) \}.
\end{displaymath}
Then,
\begin{enumerate}
\item Under the assumptions of Theorem \ref{thm_apost}, \eqref{eq_resthm} holds with $V_c(\xopt)$ in place of $V(\xopt)$.
\item Under the assumptions of Theorem \ref{thm_apriori}, \eqref{eq_resthm_apriori} holds with $V_c(\xopt)$ in place of $V(\xopt)$.
\end{enumerate}
\end{cor}
Corollary \ref{cor:cost_deter} shows that, with given confidence, the probability that $g(\xopt,\theta)$ --- hence also each agent's objective function --- deteriorates when a new realization of the uncertainty is encountered can be bounded both in an \emph{a posteriori} and an \emph{a priori} fashion as in Theorem \ref{thm_apost} and Theorem \ref{thm_apriori}, respectively. These statements are established within the proofs of Theorems \ref{thm_apost} and \ref{thm_apriori} (see \eqref{eq_coroll2}). 

\begin{rem}\label{rem:multiple_g}
	The results of Theorems \ref{thm_apost} and \ref{thm_apriori} can be adapted to the case where the uncertain part of the objective function is different for each agent, i.e., if $g$ is replaced by $g_i$, $i\in \setN$. We keep our presentation with a common $g$ since for this case we are able to construct $\Phi$ in a decentralized manner, as shown in Section \ref{sec_extgame}; the decentralized computation of $\Phi$ if the uncertain part of the objective function is different for each agent encompasses additional challenges (see \cite[Rem. 1]{FacchineiEtAl2014}) and is outside the scope of our paper. 
\end{rem}
	

%% file: solutionVI.tex
\section{Decentralized NE computation}
\label{sec_extgame}
In this section we show how to construct $\Phi$, necessary ingredient in the proof of Proposition~\ref{prop_decen}. In particular, we show that the image of $\Phi$ corresponds to the limit of a decentralized algorithm that returns a NE of the game $\game$. To achieve this, we characterise the NE of $\game$ as solutions to a variational inequality (VI) \cite{FacchineiPang2009BOOK}. We then leverage results in the literature to obtain sufficient conditions for the existence of equilibria, and set the foundations for the design of a decentralized NE computation mechanism \cite{ScutariEtAl2014}.

\subsection{VI analysis}
\label{sec_VI}
It can be observed that the presence of the $\max$ operator renders agents' objective functions \eqref{eq_obj} non-differentiable. To circumvent the computation of sub-gradients and exploit the wide range of algorithms available to solve VIs in the differentiable case, we follow the method in \cite{FacchineiEtAl2014} and define the augmented game $\gameaug$ between $N+1$ agents.
In $\gameaug$ each player $i\in\setN$, given $\xmenoi$ and $y = (y_m)_{m=1}^M$, computes
\begin{equation}\label{eq_pullout_ind}
x_i \in \argmin_{\nu_i\in\setX_i} f_i(\nu_i,\xmenoi) + \underbrace{\sum_{m=1}^M y_m g(\nu_i,\xmenoi,\theta_m)}_{\hat{g}(\nu_i,\xmenoi,y)} ,
\end{equation}
where $\hat{g}(x,y)$ follows from the equivalence, holding for any $x$,
\begin{equation}\label{eq_maxcont}
\max_{m\in\{1,\ldots,M\}} g(x,\theta_m) = \max_{y\in\Delta} \sum_{m=1}^M y_m g(x,\theta_m),
\end{equation}
where $\Delta = \left\{y\in\mathbb{R}^M\colon y\geq 0, \sum_{m=1}^M y_m = 1\right\}$ is the simplex in $\mathbb{R}^M$~\cite[Lemma~6.2.1]{RustemHoweBOOK}.
The additional agent (could be thought of as a coordinating authority), given $x$, will act instead as a maximizing player for the uncertain component of $J_i$, $i\in\setN$, i.e., 
\begin{equation}\label{eq_pullout_coord}
y \in \argmax_{\nu\in\Delta} \hat{g}(x,\nu).
\end{equation}

Note that, for any $M$-multisample $(\theta_1,\ldots,\theta_M)\in\Theta^M$, the objective functions in \eqref{eq_pullout_ind} and \eqref{eq_pullout_coord} are differentiable by Assumption~\ref{assum_eqNE_VIsol}. We now link the NE of the augmented game $\gameaug$ to a VI.
To this end, we define the mapping $F(x,y)\colon \setX \times \Delta \to \mathbb{R}^{(nN + M)}$ as the pseudo-gradient~\cite[\S 1.4.1]{FacchineiPang2009BOOK}
\begin{equation}\label{eq_F}
F(x,y) = \left[
\begin{array}{c}
\left(\nabla_{x_i}f_i(x) + \nabla_{x_i}\hat{g}(x,y)\right)_{i\in\setN} \\[6pt]
-\left(\nabla_{y_m}\hat{g}(x,y)\right)_{m=1}^M
\end{array}\right].
\end{equation}

Letting $z = (x,y)$, the VI problem takes the form \cite[\S 1.4.2]{FacchineiPang2009BOOK}
\begin{equation}
\label{eq_NEVI}
\begin{split}
\mathrm{find} \;  & {}\zopt \in \setX\times\Delta\\
\mathrm{subject~to} \;  & {}(z-\zopt)\trasp F(\zopt) \geq 0,\; \forall z\in\setX\times\Delta.
\end{split}
\end{equation}
The constraints in \eqref{eq_NEVI}
represent the concatenation of the first-order optimality conditions for the $N+1$ individual problems described by~\eqref{eq_pullout_ind} and \eqref{eq_pullout_coord}. In the following, we refer to the problem described by \eqref{eq_NEVI} as VI$(F,\setX\times\Delta)$. 

While the VI$(F,\setX\times\Delta)$ describes (under the conditions in Assumption~\ref{assum_eqNE_VIsol}) the NE of $\gameaug$, it turns out that the former can be also linked to the equilibria of $\game$, as formalized next. 


\begin{prop}\label{prop_eqNE_VIsol}
Under Assumption~\ref{assum_eqNE_VIsol}, there always exists a solution of VI$(F,\setX\times\Delta)$. Denote such a solution by $\zopt = (\xopt,\yopt)$. We then have that $\xopt$ is a NE of $\game$.
\end{prop}
\begin{proof}
The existence of a solution for the VI$(F,\setX\times\Delta)$ is guaranteed by \cite[Cor.~2.2.5]{FacchineiPang2009BOOK} under Assumption~\ref{assum_eqNE_VIsol} and the compactness of $\Delta$. Denote such a solution by $\zopt$. 
A link between the solutions of the VI and those of the augmented game is established by \cite[Prop.~1.4.2]{FacchineiPang2009BOOK}: $\zopt = (\xopt,\yopt)$ is a solution of $\gameaug$ if and only if it solves VI$(F,\setX\times\Delta)$. The link with the original game $\game$ is provided by \cite[Thm.~1]{FacchineiEtAl2014}: for any NE $(\xopt,\yopt)$ of the game $\gameaug$, $\xopt$ is a NE of $\game$, which concludes the proof. 
\end{proof}
%
\subsection{Monotonicity of the augmented VI operator}
\label{sec_monotonF}
The development of algorithms for the solution of VI problems relies upon the monotonicity of the mapping $F$ in~\eqref{eq_NEVI}, which plays a role analogous to convexity in optimization~\cite{ScutariEtAl2014}.
\begin{defn}[Monotonicity]\label{defn_mono}
	A mapping $F\colon \setD \to \mathbb{R}^{m}$, with $\setD\subseteq\mathbb{R}^{m}$ closed and convex, is
	\begin{itemize}
		\item monotone on $\setD$ if $(u-v)\trasp(F(u)-F(v))\geq 0$ for all $u,v\in\setD$,
		\item strongly monotone on $\setD$ if there exists $c>0$ such that $(u-v)\trasp(F(u)-F(v))\geq c\|u-v\|^2$ for all $u,v\in\setD$.
	\end{itemize}
\end{defn}
\vspace{0.2cm}
The following result is instrumental in our analysis:
\begin{lem}\label{lem_Fmono}
	Let Assumptions~\ref{assum_eqNE_VIsol} hold. Then 
$F(x,y)$ in~\eqref{eq_F} is monotone on $\setX\times\Delta$.
\end{lem}
\begin{proof}
	By Assumption~\ref{assum_eqNE_VIsol}(ii), $F(x,y)$ is continuously differentiable on its domain.
	Let $F^x$ and $F^y$ denote the first $nM$ and the last $M$ rows of $F$, respectively, i.e., $F^{x}(x,y) = \left(\nabla_{x_i}f_i(x) + \nabla_{x_i}g(x,y)\right)_{i\in\setN}$, and $F^{y}(x,y) = \left(\nabla_{y_m}g(x,y)\right)_{m=1}^M$. 
	By definition of the Jacobian we have
	\begin{multline}\label{eq_Jacobbe}
	J F(x,y) = \left[
	\begin{array}{cc}
	J_{x}F^{x}(x,y) & J_{y}F^{x}(x,y) \\
	-J_{x}F^{y}(x,y) & -J_{y}F^{y}(x,y)
	\end{array}\right] 
	\end{multline}
	where $J_x F^x(x,y) = (\nabla^2_{x_i x_j} f_i(x))_{i,j\in\setN} + \nabla^2_{xx} \hat{g}(x,y)$, $J_y F^x = (J_x F^y(x,y)) \trasp = (\nabla^2_{x_i y} (f_i(x) + \hat{g}(x,y)))_{i\in\setN}$, and $J_{y}F^{y}(x,y) = 0$; notice that $(\nabla^2_{x_i x_j} f_i(x))_{i,j \in \setN} $ is a matrix with $\nabla^2_{x_i x_j} f_i(x)$ being its $(i,j)$-th entry.\\
	Due to the particular block structure in \eqref{eq_Jacobbe},  $J F(x,y) \succeq 0$ if and only if $J_{x}F^{x}(x,y) \succeq 0$.
	To show this, note that by from \eqref{eq_assum2}, for all $(x,y)\in\setX\times\Delta$ and all $\nu\in\mathbb{R}^{nN}$,
	\begin{equation}
	\begin{split}
	\nu\trasp (\nabla^2_{x_i x_j} f_i(x))_{i,j \in \setN}~ \nu & \geq \chi^f \|\nu\|^2,\\
	\nu\trasp \nabla^2_{x x} \hat{g}(x,y) \nu & \geq \chi^g \|\nu\|^2. 
	\end{split}
	\end{equation}
	Summing the above inequalities yields 
	\begin{multline}\label{eq_psd3}
	\nu\trasp \left((\nabla^2_{x_i x_j} f_i(x))_{i,j \in \setN} + \nabla^2_{xx} \hat{g}(x,y) \right ) \nu \\
	= \nu\trasp J_{x}F^{x}(x,y) \nu \geq (\chi^f + \chi^g) \|\nu\|^2,\; \forall \nu\in\mathbb{R}^{nN},
	\end{multline}
	which, since $\chi^f + \chi^g \geq 0$, corresponds to $J_{x}F^{x}(x,y) \succeq 0$. 
	The statement then follows directly from~\cite[Prop.~2.3.2]{FacchineiPang2009BOOK}, thus concluding the proof.
\end{proof}
A direct consequence of the monotonicity of $F$ is that by \cite[Thm.~41]{ScutariEtAl2014}, VI$(F,\setX\times\Delta)$ may admit multiple solutions: this fact together with \cite[Prop.~1.4.2]{FacchineiPang2009BOOK} --- stating the correspondence between the solutions of the  VI$(F,\setX\times\Delta)$ and the NE of $\gameaug$ --- implies that the game $\gameaug$ can admit multiple NE.

\subsection{Decentralized algorithm for monotone VI and equilibrium selection}
\label{sec_decmapping}
Specific algorithmic design is needed to tackle the convergence properties of a monotone mapping; we refer the reader to \cite{KannanShanbhag2012SIAM,ScutariEtAl2014,YiPavel2019,BelgioiosoGrammatico2020} for a deep discussion on this topic. For our scope, proximal algorithms can be used to retrieve a solution of a monotone VI by solving a particular sequence of strongly monotone problems, derived by regularizing the original problem.
To construct a decentralized mapping $\Phi$ that is single-valued, as required by Proposition \ref{prop_decen}, a \emph{tie-break} rule needs to be put in place to single a particular NE out of the possibly many (see Section \ref{sec_monotonF}).
Such a tie-break rule is needed even if only one NE is returned by the given algorithm, to prevent the case where different initial conditions produce different NE.

We address the above by employing a proximal algorithm based on \cite[Algorithm~4]{ScutariEtAl2014}, which allows us to select the minimum Euclidean norm NE; the choice of the Euclidean norm is not restrictive, and a wide range of strictly convex objective function can be used as a selector instead (see \cite[Thm.~21]{ScutariEtAl2014}). Thus, more formally, we consider the following refinement of \eqref{eq_NEVI}
\begin{subequations}\label{eq_NEVIsel}
	\begin{align}
	\zopt =  \;  & {}\argmin_{\zopt \in\setX\times\Delta} \frac{1}{2} \|\zopt\|^2 \\
	& {} \mathrm{subject~ to} \;  (z-\zopt)\trasp F(\zopt) \geq 0,\; \forall z\in\setX\times\Delta.\label{eq_NEVIsel_b}
	\end{align}
\end{subequations}

We link the VI problem in \eqref{eq_NEVIsel} to the regularized game $\gameaug^{\tau,\bar{z}}$, where $\tau \in \mathbb{R}_+$ and $\bar{z}=(\bar{x},\bar{y})$ are the designated step size and centre of regularization, respectively.
Given the tuple $(\xmenoi,y,\bar{x}_i)$, each player $i\in\setN$ solves the following problem
\begin{multline}\label{eq_pullout_ind_tau}
	x_i = \argmin_{\nu_i\in\setX_i} f_i(\nu_i,x_{-i}) + \hat{g}(\nu_i,x_{-i},y)\\
					+ \frac{\eta}{2}\|(\nu_i,\xmenoi,y)\|^2 + \frac{\tau}{2}\|\nu_i-\bar{x}_i\|^2,
\end{multline}
while the additional agent (player $N+1$), given $(x,\bar{y})$, solves
\begin{equation}\label{eq_pullout_coord_tau}
	y = \argmax_{\nu\in\Delta} \hat{g}(x,\nu) - \frac{\eta}{2} \|(x,\nu)\|^2 - \frac{\tau}{2}\|\nu-\bar{y}\|^2,
\end{equation}
with $\eta\in\mathbb{R}_+$. 
Note that Assumption~\ref{assum_eqNE_VIsol} still holds for~\eqref{eq_pullout_ind_tau}--\eqref{eq_pullout_coord_tau}.
By taking the pseudo-gradient of the above as in~\eqref{eq_F}, we have from \cite[Prop.~1.4.2]{FacchineiPang2009BOOK} that $\zopt = (\xopt,\yopt)$ is a NE of $\gameaug^{\tau,\bar{z}}$ if and only if it satisfies the VI
\begin{equation}\label{eq_NEVItau}
	(z-\zopt)\trasp (F(\zopt) + \eta \zopt + \tau (\zopt-\bar{z})) \geq 0,\; \forall z,\bar{z}\in\setX\times\Delta.
\end{equation}
The next lemma shows that the regularized game admits a unique NE.
\begin{lem}\label{lem_strongmonotontau}
	Consider Assumptions~\ref{assum_eqNE_VIsol}. Let $F$ be as in~\eqref{eq_F}, and fix $\eta\geq 0$. 
	Then, for any $\tau>0$ and $\bar{z}\in\setX\times\Delta$, the regularized game $\gameaug^{\tau,\bar{z}}$ defined by~\eqref{eq_pullout_ind_tau}--\eqref{eq_pullout_coord_tau} admits a unique NE.
\end{lem}
\begin{proof}
To establish uniqueness of the NE it suffices to show that $F^{\tau,\bar{z}}$ is strongly monotone \cite[Thm.~41]{ScutariEtAl2014}.
	Fix any $u,v \in \mathbb{R}^{nN+M}$. Let $F^{\tau,\bar{z}}(u) = F(u) +\eta u +\tau(u-\bar{z})$, and define $F^{\tau,\bar{z}}(v)$ similalrly. We have 
	\begin{align}
	(u-v)&\trasp(F^{\tau,\bar{z}}(u)-F^{\tau,\bar{z}}(v)) \nonumber \\
	&= (u-v)\trasp ( F(u) - F(v) ) + (\eta+\tau) \| u-v\|^2 \nonumber \\
	&\geq \tau \| u-v\|^2, \label{eq:strong_mon}
	\end{align}
	where the inequality follows from the fact that $\eta \in \mathbb{R}_+$, and $(u-v)\trasp ( F(u) - F(v) ) \geq 0$ since $F$ is monotone due to Lemma \ref{lem_Fmono}. 
	By Definition \ref{defn_mono}, \eqref{eq:strong_mon} implies that $F^{\tau,\bar{z}}$ is strongly monotone, thus concluding the proof. \end{proof}

Now let $S^{\tau}(\cdot)$ denote the solution of the VI$(F^{\tau,\cdot},\setX\times\Delta)$.
Building on Lemma~\ref{lem_strongmonotontau}, we aim at determining a NE of $\game$ by updating the centre of regularization of $\gameaug^{\tau,\cdot}$ on the basis of an iterative method in the form $\bar{z}^{(k+1)} = S^{\tau}(\bar{z}^{(k)})$, until convergence to the fixed point $\zopt = S^{\tau}(\zopt)$. The latter corresponds to the (unique under Lemma \ref{lem_strongmonotontau}) NE of $\gameaug^{\tau,\zopt}$, which satisfies~\eqref{eq_NEVIsel}. Algorithm~\ref{alg_iterprox} provides the means to establish such a connection; this is formalised in the following proposition.
\begin{algorithm}[t]
	\caption{Decentralized NE seeking algorithm}
	\label{alg_iterprox}
	\begin{algorithmic}[1]
		\small
		\Require $\bar{z}^{(0)}\in\setX\times\Delta$, $\tau>0$, $\gamma_{\mathrm{inn}}>0$, $\gamma_{\mathrm{out}}>0$
		\vspace{0.1cm}
		\State $k \gets 0$
		\Repeat ~(outer loop)
		\vspace{0.1cm}
		\State $l \gets 0$
		\Repeat ~(inner loop)
		\For{$i = 1,\ldots,N$}
		\vspace{0.1cm}
		\State $ x_i^{(l+1)} \gets \argmin\limits_{{\nu_i\in\setX_i}} f_i(\nu_i,x_{-i}^{(l)}) + \hat{g}(\nu_i,x_{-i}^{(l)},y^{(l)})$ \par 
		\hspace{6em} $ {} + \frac{\eta^{(k)}}{2}\|(\nu_i,x_{-i}^{(l)},y^{(l)})\|^2 + \frac{\tau}{2} \|\nu_i - \bar{x}_i^{(k)}\|^2$ 
		\EndFor
		\vspace{0.15cm}
		\State $y_i^{(l+1)} \gets \argmax\limits_{{\nu\in\Delta}} \hat{g}(x^{(l)},\nu)$ \par
		 \hspace{7em} $ {} - \frac{\eta^{(k)}}{2}\|(x^{(l)},\nu)\|^2 - \frac{\tau}{2} \|\nu - \bar{y}^{(k)}\|^2$
		 \vspace{0.1cm}
		\State $l \gets l+1$
		\vspace{0.1cm}
		\Until{$\|{z}^{(l)}-{z}^{(l-1)}\|\leq \gamma_{\mathrm{inn}}$}
		\vspace{0.1cm}
		\State $\bar{z}^{(k+1)}\gets z^{(l)}$
		\vspace{0.1cm}
		\State $k \gets k+1$
		\vspace{0.1cm}
		\Until{$\|\bar{z}^{(k)}-\bar{z}^{(k-1)}\|\leq \gamma_{\mathrm{out}}$} %
	\end{algorithmic}
\end{algorithm}

\begin{prop}[Thm.~21 \cite{ScutariEtAl2014}]\label{lem_convalg}
	Let Assumptions~\ref{assum_eqNE_VIsol} hold. 
	Let $\{\eta^{(k)}\}_{k=0}^{\infty}$ be any sequence satisfying $\eta^{(k)}>0$ for all $k$, $\sum_{k=0}^{\infty} \eta^{(k)} = \infty$, and $\lim_{k\rightarrow\infty} \eta^{(k)} = 0$. Select a big enough $\bar{\tau}>0$ and let $\{\bar{z}^{(k)}\}_{k=0}^{\infty}$ denote the sequence generated by Algorithm \ref{alg_iterprox}.
	For any $\tau\geq\bar{\tau}$, (i) there exists $\gamma_{\mathrm{inn}},\gamma_{\mathrm{out}}>0$ such that $\{\bar{z}^{(k)}\}_{k=0}^{\infty}$ is bounded; (ii) there exists $\zopt = (\xopt,\yopt)$ such that $\|\bar{z}^{(k)} - \zopt\| \rightarrow 0$ for $k\rightarrow\infty$, where $\zopt$ is a solution of \eqref{eq_NEVIsel}, and (iii) $\xopt \in \Omega$. 
\end{prop}

The reader is referred to \cite[Lemma.~20]{ScutariEtAl2014} for a lower bound on $\bar{\tau}$ in Algorithm~\ref{alg_iterprox}. The latter asymptotically converges to a solution of \eqref{eq_NEVIsel}, while by Proposition~\ref{prop_eqNE_VIsol}, \eqref{eq_NEVIsel_b} it is equivalent to the game $\gameaug$, whose solution set is nonempty and is also contained in $\Omega$ due to the second part of Proposition~\ref{prop_eqNE_VIsol}.

\emph{Proof of Proposition \ref{prop_decen}:}
Algorithm~\ref{alg_iterprox} and its analysis, leading to Proposition~\ref{lem_convalg}, serves as an implicit construction of a decentralized, single-valued mapping $\Phi\colon \Theta^M \to \Omega$, thus establishing Proposition \ref{prop_decen}. \hfill $\blacksquare$

Note that $\Phi$ is single-valued, hence the returned solution is independent of the initial condition used in Algorithm~\ref{alg_iterprox}. However, in the proof of Theorem \ref{thm_apriori}
it becomes insightful to make this dependency explicit. Thus, for the analysis of Section \ref{sec_thmproof_new} we will introduce the notation $\Phi_{x_0}$, with $x_0 = \bar{x}^{(0)} \in \setX$. 
(Notice that $\bar{y}^{(0)}$, which also appears in the initial condition $\bar{z}^{(0)}$ of Algorithm~\ref{alg_iterprox}, depends on the $M$-multisample; as the latter is already an argument of $\Phi$, we only include $x_0$ as a subscript.)

\section{Proofs of \emph{a posteriori} and \emph{a priori} certificates}
\label{sec_thmproof}
\subsection{Proof of Theorem~\ref{thm_apost}}
\label{sec_thmproof_apost}
Fix any $M\in \mathbb{N}$. Consider $(\theta_1,\ldots,\theta_M)\in\Theta^M$, and let $d^*\leq M$ be the cardinality of any given compression set of $\{\theta_1,\ldots,\theta_M\}$ (recall that it depends on the observation of the $M$-multisample).
Let $\xopt = \Phi(\theta_1,\ldots,\theta_M) \in \Omega$, and
	\begin{equation}
		\gammaopt = \max_{m \in \{1,\ldots,M\}} g(\xopt,\theta_m).
	\end{equation}
For any $\theta\in\Theta$, consider the set
\begin{equation}
	H_{\theta} = \{(x,\gamma)\colon g(x,\theta)\leq \gamma \}.
\end{equation}

Fix $\beta \in (0,1)$ and consider $\varepsilon(\cdot)$ defined as in \eqref{eq_eps}. Under Assumptions \ref{assum_eqNE_VIsol}, $\Phi$ is single-valued by Proposition \ref{prop_decen}. By \cite[Thm.~1]{CampiEtAl2018TAC} we then have that
	\begin{multline}\label{eq_coroll2}
	\mathds{P}^M\big\{(\theta_1,\ldots,\theta_M) \in \Theta^M:\\
	\mathds{P}\{\theta\in\Theta\colon (\xopt,\gammaopt)\notin H_{\theta} \}  \leq \varepsilon(d^*) \big\}\geq 1-\beta,
	\end{multline}
	if the following consistency condition holds for $H_{\theta}$ (see \cite{MargellosEtAl2015TAC} for a definition)
	\begin{equation} \label{eq:consistency}
	(\xopt,\gammaopt)\in H_{\theta_m},\; \forall m\in\{1,\ldots,M\}.
	\end{equation}
To show the latter, notice that for each $i\in\setN$, by the NE definition (Definition \ref{defn_NE}), $(\xopt_i,\gammaopt)$ will belong to the set of minimizers of the following epigraphic reformulation of \eqref{eq_defNE}
	\begin{subequations}\label{eq_epigame}
		\begin{align}
		(\xopt_i,\gammaopt)\in & \argmin_{x_i \in \setX_i,\gamma \in \mathbb{R}} \, f_i(x_i,\xmenoi^{\ast}) + \gamma \\
		\mathrm{subject~ to~} &g(x_i,\xmenoi^{\ast},\theta_m)\leq \gamma, \; \forall m \in\{1,\cdots,M\}.\label{eq_constr_max}
		\end{align}
	\end{subequations}
	By \eqref{eq_constr_max} it follows then that \eqref{eq:consistency} is satisfied, thus establishing \eqref{eq_coroll2}.
Note that for the result of \cite{CampiEtAl2018TAC} to be invoked, \eqref{eq_epigame} is not required to be a convex optimization program, hence the fact that for each $i\in \setN$, for any $\theta \in \Theta$, only $f_i(\cdot,\xopt_{-i})+g(\cdot,\xopt_{-i},\theta)$ is assumed to be convex by Assumption~\ref{assum_eqNE_VIsol} is sufficient.

	By the definition of $\gammaopt$ and $H_{\theta}$, \eqref{eq_coroll2} implies that with confidence at least $1-\beta$, $\mathds{P}\{\theta\in\Theta\colon g(\xopt,\theta) > \max_{m \in \{1,\ldots,M\}} g(\xopt,\theta_m) \} \leq \varepsilon(d^*)$, thus establishing the first part of Corollary \ref{cor:cost_deter}.
	
	We now proceed to demonstrate the claim in \eqref{eq_resthm}. 
	Recall that, by \eqref{eq_F}, \eqref{eq_NEVIsel} and Proposition~\ref{prop_eqNE_VIsol}, we can obtain $\xopt\in\Omega$ as solution of the following optimization program (note the slight abuse of notation as by $(\xopt,\yopt)$ we denote both the optimizer and the corresponding decision vector)
	\begin{subequations}\label{eq_selNE2}
		\begin{align}
		& \min_{(\xopt,\yopt) \in \setX\times\Delta}  \; \frac{1}{2} \|(\xopt,\yopt)\|^2  \\
		& \mathrm{subject~ to} \nonumber \\
		& \sum_{i\in\setN}(x_i-\xopt_i)\trasp \nabla_{x_i}(f_i(\xopt)+\hat{g}(\xopt,\yopt)) \nonumber \\
		- &\sum_{m=1}^M (y_m-\yopt_m) \nabla_{y_m} \hat{g}(\xopt,\yopt) \geq 0,\,
		\forall x\in\setX,\,y\in\Delta,\label{eq_constrScn2}
		\end{align}	
	\end{subequations}
	where $(\xopt,\yopt)$ is a NE of $\gameaug$. By definition of $\hat{g}$ in~\eqref{eq_pullout_ind}, and recalling $\nabla_{y_m} (\sum_{m=1}^M \yopt_m g(\xopt,\theta_m) ) = g(\xopt,\theta_m)$, \eqref{eq_constrScn2} can be equivalently written as
	\begin{multline}\label{eq_constrScn2bis}
	\sum_{i\in\setN} (x_i-\xopt_i)\trasp \nabla_{x_i}(f_i(\xopt)+\sum_{m=1}^M \yopt_m g(\xopt,\theta_m)) \\
	+ \sum_{m=1}^M \yopt_m g(\xopt,\theta_m) - \sum_{m=1}^M y_m  g(\xopt,\theta_m) \geq 0,\\ \forall x\in\setX,\,y\in\Delta.
	\end{multline}
	As \eqref{eq_constrScn2bis} holds for all $y\in \Delta$, we have that \eqref{eq_constrScn2} is equivalent to the following inequality being satisfied for all $x \in \mathcal{X}$, 
	\begin{equation}\label{eq_constrScn3bis}
	\begin{split}
	\sum_{i\in\setN} (x_i-{}&\xopt_i)\trasp \nabla_{x_i}(f_i(\xopt)+\sum_{m=1}^M \yopt_m g(\xopt,\theta_m)) \\
	& + \sum_{m=1}^M \yopt_m g(\xopt,\theta_m) - \max_{y\in\Delta}\sum_{m=1}^M y_m  g(\xopt,\theta_m) \\
	= \sum_{i\in\setN} (x_i{}&-\xopt_i)\trasp \nabla_{x_i}(f_i(\xopt)+\sum_{m=1}^M \yopt_m g(\xopt,\theta_m)) \\
	& + \sum_{m=1}^M \yopt_m  g(\xopt,\theta_m) 
	- \max_{m\in\{1,\ldots,M\}}g(\xopt,\theta_m)\geq 0,
	\end{split}
	\end{equation}
	where the equality follows from \eqref{eq_maxcont}.
	For a given $\theta\in\Theta$, recall from Section~\ref{sec_thmstate} the definition of the game $\game^+$ associated with the samples $\{\theta_1,\ldots,\theta_M\}\cup \{\theta\}$, and the associated set of NE $\Omega^+$. Moreover, let $\gameaug^+$ denote the associated augmented game. Analogously to \eqref{eq_constrScn3bis}, any solution $(x^+,y^+)\in\setX\times\Delta^+$, where $\Delta^+$ is the simplex in $\mathbb{R}^{M+1}$, of the augmented game $\gameaug^+$ will satisfy the following VI:
	\begin{multline}\label{eq:eq_constrScn_p}
		\sum_{i\in\setN} (x_i -x_i^+)\trasp \nabla_{x_i}\big(f_i(x^+)+\sum_{m=1}^{M} y_m^+ g(x^+,\theta_m)\big)\\
		 + \sum_{i\in\setN} (x_i -x_i^+)\trasp \nabla_{x_i}\big(y_{M+1}^+ g(x^+,\theta)\big) \\
		+ \sum_{m=1}^{M} y_m^+  g(x^+,\theta_m) + y_{M+1}^+ g(x^+,\theta)\\ 
		- \max\big\{\max_{m\in\{1,\ldots,M\}}g(x^+,\theta_m),\, g(x^+,\theta)\big\}\geq 0.
	\end{multline}
	Note the analogy between \eqref{eq:eq_constrScn_p} and \eqref{eq_constrScn3bis}, with the additional terms corresponding to the new sample $\theta$ ($y_{M+1}$ is the additional decision variable corresponding to the new sample).
	
	We are interested in quantifying the probability of $\xopt \in \Omega^+$. To this end, notice that if $g(\xopt,\theta) \leq \gammaopt$, then
	$x^+ = \xopt$ and $y^+ = (y^{\ast\traspo},\,0)\trasp$ constitute a feasible pair for \eqref{eq:eq_constrScn_p}. This is due to the fact that under this choice $y_{M+1}^+ = 0$ and hence
	\begin{multline}
		\max\Big\{\max_{m\in\{1,\ldots,M\}}g(\xopt,\theta_m),\, g(\xopt,\theta)\Big\} \\
		= \max\big\{\gammaopt,\, g(\xopt,\theta)\big\}  = \gammaopt = \max_{m\in\{1,\ldots,M\}}g(\xopt,\theta_m),\nonumber
	\end{multline}
	thus
\eqref{eq:eq_constrScn_p} reduces to \eqref{eq_constrScn3bis}. Applying Proposition~\ref{prop_eqNE_VIsol} to $\game^+$ and $\gameaug^+$, we have that if $(x^+,y^+)$ satisfies \eqref{eq:eq_constrScn_p} (i.e., it is a NE of the augmented game $\gameaug^+$) then $x^+ \in\Omega^+$. 	
Therefore, $\xopt \in \Omega^+$ whenever $g(\xopt,\theta)\leq \gammaopt$, or in other words
	\begin{multline}\label{eq_probbound}
		\mathds{P}\{\theta\in\Theta\colon (\xopt,\gammaopt)\in H_{\theta} \}  = \mathds{P}\{\theta\in\Theta\colon g(\xopt,\theta) \leq \gammaopt \} \\
		  \leq \mathds{P}\{\theta\in\Theta\colon \xopt \in \Omega^+ \}.
	\end{multline}
	By \eqref{eq_coroll2} and \eqref{eq_probbound}, \eqref{eq_resthm} follows, thus concluding the proof. \hfill $\blacksquare$

\subsection{Proof of Theorem~\ref{thm_apriori}}
\label{sec_thmproof_new}
Let $\setC_0 \subseteq \{\theta_1,\ldots,\theta_M\}$ be the minimal cardinality compression set for the minimum norm NE $\xopt$ returned by Algorithm \ref{alg_iterprox}; note that under the non-degeneracy assumption it will be unique and it will coincide with the set of support samples. 
Following the discussion at the end of Section~\ref{sec_extgame}
$\Phi_{x_0}$ denotes a mapping that returns $\xopt$ (e.g., the one induced by Algorithm \ref{alg_iterprox}), where we make explicit the dependence on the initial condition $x_0 \in \setX$.
As $\Phi$ is single-valued, by definition of a compression set we have that $\xopt = \Phi_{x_0}(\theta_1,\ldots,\theta_M) = \Phi_{x_0}(\setC_0)$, for all $x_0 \in \setX$. 

Consider now the following optimization program.
	\begin{subequations}\label{eq_epigame_sum}
		\begin{align}
		(\xopt_i, \gammaopt_i)_{i \in \setN} = & \argmin_{\{x_i \in \setX_i, \gamma_i \in \mathbb{R}\}_{i \in \setN}} \, \sum_{i \in \setN} \big ( \gamma_i + \tau \|x_i - \xopt_i\|^2 \big ) \label{eq:obj_sum}\\
		\mathrm{subject~ to~ } 
		& f_i(x_i,\xmenoi^{\ast}) +  g(x_i,\xmenoi^{\ast},\theta_m)\leq \gamma_i, \nonumber \\ 
		&~~~~~~~~\forall i\in \setN,~ \forall m \in\{1,\cdots,M\},
		\end{align}
	\end{subequations}
	where $\gamma_i$, $i\in\setN$, are epigraphic variables, and in \eqref{eq:obj_sum} we have equality and not inclusion since the set of minimizers is a singleton due to the regularization term $\tau \|x_i - \xopt_i\|^2$. 
	Note that \eqref{eq_epigame_sum} is separable across $i \in \setN$, with each subproblem corresponding to an epigraphic reformulation of the fixed point characterization of the regularized problem \eqref{eq_pullout_ind_tau} for $\eta =0$ (this is the value of $\eta^{(k)}$ when the proposed algorithm has converged to $\xopt$) and $\bar{x}_i = \xopt_i$, for all $i \in \setN$. 
	The latter is identical to $\Phi_{x_0}$ with $x_0 = \xopt$. 
	It follows from \eqref{eq_epigame_sum} that $\xopt = \Phi_{\xopt}(\theta_1,\ldots,\theta_M)$.
Also note that we have introduced one epigraphic variable per agent $i \in \setN$; we will invoke in the sequel the fact that \eqref{eq_epigame_sum} is convex due to Assumption \ref{assum_eqNE_VIsol} . However,
if we further assume that for all $i\in\setN$, for every fixed $x_{-i}\in\setX_{-i}$ and $\theta\in\Theta$, the functions $f_i(\cdot,x_{-i})$ and $g(\cdot,x_{-i},\theta)$ are each convex, we would only need one epigraphic variable, as we could perform an epigraphic reformulation only for $g$; this would give rise to the constraint in \eqref{eq_constr_max}, which is common to all agents. 

 Let $\setC$ denote a minimal cardinality compression set for $\xopt$ in \eqref{eq_epigame_sum}. We claim that $\setC_0 \subseteq \setC$. To show this, assume for the sake of contradiction that there exists $k \in \{1,\ldots,M\}$ such that $\theta_k \in \setC_0$ but $\theta_k \notin \setC$. Consider the set $\{\theta_1,\ldots,\theta_M\} \setminus \{\theta_k\} \supseteq \setC$, and notice that this has to be a compression set for $\xopt$ in \eqref{eq_epigame_sum} as it is a superset of $\setC$. By Definition \ref{defn_CS}, this implies that $\xopt = \Phi_{\xopt}(\{\theta_1,\ldots,\theta_M\} \setminus \{\theta_k\})$ (recall that the solution of \eqref{eq_epigame_sum} is given by $\xopt = \Phi_{\xopt}(\theta_1,\ldots,\theta_M)$). However, $\theta_k \in \setC_0$; as the latter coincides with the set of support samples due to the imposed non-degeneracy assumption, we have by Definition \ref{defn_suppsample} $\xopt \neq \Phi_{\xopt}(\{\theta_1,\ldots,\theta_M\} \setminus \{\theta_k\})$. This establishes a contradiction, showing that $\setC_0 \subseteq \setC$ (hence $|\setC_0| \leq | \setC |$).

By Assumptions \ref{assum_eqNE_VIsol}, \eqref{eq_epigame_sum} is a convex scenario program, and admits a unique solution due to the fact that the objective function in \eqref{eq:obj_sum} is strictly convex. Moreover, it has a non-empty feasibility region in view of Proposition \ref{prop_eqNE_VIsol}. Therefore, by  \cite{CampiGaratti2008SIAM}, \cite{Calafiore2010SIAM}, we have that any minimal cardinality compression set $\setC$ has cardinality upper-bounded by $(n+1)N$, i.e., the number of decision variables in \eqref{eq_epigame_sum}. Therefore, $|\setC_0| \leq | \setC | \leq (n+1)N$. As a result, $|\setC_0|$ can be upper-bounded by the \emph{a priori} known quantity $(n+1)N$.
As Theorem \ref{thm_apost} holds for any compression cardinality $d^* \geq |\setC_0|$, we can apply it with $d^* = (n+1)N$. Hence, Theorem \ref{thm_apriori} as well as the second part of Corollary \ref{cor:cost_deter} directly follow, concluding the proof. \hfill $\blacksquare$

\section{Computation of the compression set cardinality}
\label{sec_apost_eval}
The result of Theorem \ref{thm_apost} relies on the computation of the compression cardinality $d^*$, which by Definition~\ref{defn_CS} is bounded by $M$.  An \emph{a posteriori} estimate of the compression cardinality can be obtained through different methodologies, whose design may be tuned on the specific case. It follows from \eqref{eq_resthm} that the closer the estimate to the minimal cardinality of the compression sets in $\mathfrak{C}$, the less conservative the probabilistic guarantees on the robustness performance of the solution.	 
	In \cite[\S II]{CampiEtAl2018TAC} a greedy procedure is outlined to estimate (an upper bound to) the minimal compression cardinality; for completeness, we summarize this procedure in Algorithm~\ref{alg_suppconstenum}.
	According to this, a compression set $\setC$ is constructed progressively by removing samples one by one (step 2). Only if their removal leaves the solution unaltered they are discarded (step 3-4); this is then repeated till no further sample can be removed without changing $\xopt$.
	\begin{algorithm}[t]
		\caption{Greedy computation of compression set}
		\label{alg_suppconstenum}
		\begin{algorithmic}[1]
			\small
			\Require $\setC = \{\theta_1,\ldots,\theta_M\}$, $\xopt \in \Omega$;
			\ForAll{$m\in\{1,\ldots,M\}$}
			\State $\xcirc \gets \Phi(\setC\setminus\theta_m)$;
			\If{$\xcirc = \xopt$}
			\State $\setC \gets \setC\setminus\theta_m$;
			\EndIf
			\EndFor
		\end{algorithmic}
	\end{algorithm}
	However, there are two drawbacks: first, the computational cost is generally high, as $\Phi(\cdot)$ should be evaluated at least $M$ times, where each of these operations typically involves an asymptotic scheme (as, e.g., in Algorithm \ref{alg_iterprox}); second, in practice, limited numerical accuracy makes the evaluation of the condition at Step 3 of Algorithm~\ref{alg_suppconstenum} amenable to numerical errors. 

To alleviate these, we provide a computationally efficient way to determine a compression set, and hence $d^*$, by direct inspection of the NE.
To achieve this, we impose certain NE uniqueness requirements. However, it should be noted that for the wide class of aggregative games, the additional structure required in the proposition below 
implies only uniqueness of an aggregate strategy, where multiple equilibria may exist.
This is summarized in the following proposition.

\begin{prop}\label{prop_apost_eval}
	Consider Assumption~\ref{assum_eqNE_VIsol}. Further assume that for all $M \in \mathbb{N}$, either
		\begin{enumerate}
	\item $\game$ admits a unique NE;
	\item or, $g$ depends on the aggregate strategy\footnote{With a slight abuse of notation, in the second part of the proposition it is to be understood that for all $i\in\setN$ and for any given $x_{-i}$, Assumption~\ref{assum_eqNE_VIsol} refer to the function $f_i(\cdot,x_{-i}) + g(\sigma(\cdot,x_{-i}),\theta)$. } $\sigma(x)\colon x \mapsto \sum_{i\in\setN}x_i$, and $\game$ admits a unique  NE aggregate $\sigma(x)$.
	\end{enumerate}
	Then, the set $\setY^{\ast}\triangleq \{m\in\{1,\ldots,M\}\colon \yopt_m > 0\}$ corresponds to the indices of a compression set, i.e., $\xopt = \Phi(\theta_1,\ldots,\theta_M) = \Phi(\{\theta_m\}_{m\in \setY^{\ast}})$.
\end{prop}
\begin{proof}
\emph{Part 1: Uniqueness of NE. } 
	Fix any $(\theta_1,\ldots,\theta_M) \in \Theta^M$ and notice that it forms a (trivial) compression set for $\xopt$. Let $(\xopt,\yopt)$ be a solution of the augmented game $\gameaug$, where $\yopt = (\yopt_m)_{m=1}^M$. 
	
	To prove that $\xopt = \Phi(\{\theta_m\}_{m\in \setY^{\ast}})$ it suffices to show that the solution returned by $\Phi$ remains unaltered after removing all samples from $\{\theta_1,\ldots,\theta_M\}$ whose associated component of $\yopt$ is zero.
 To this end, suppose that at least one such sample exists: without loss of generality, assume $\yopt_M = 0$ (i.e., that sample has index $M$). We will first show that $\{\theta_1,\ldots,\theta_{M-1}\}$ is a compression set, i.e., $\xopt = \Phi(\theta_1,\ldots,\theta_{M-1})$.
	Let
	$\game^- = \langle \setN,(\setX_i)_{i\in\setN}, (J_i)_{i\in\setN},\{ \theta_j \}_{j=1}^{M-1} \rangle$ be the game with samples $\{\theta_1,\ldots,\theta_{M-1}\}$. Moreover, let $\gameaug^-$ denote the associated augmented game, and $\Delta^-$ the simplex in $\mathbb{R}^{M-1}$. 
	Since $(\xopt,\yopt)$ is an NE of $\gameaug$, it will satisfy the VI in \eqref{eq_constrScn3bis}. At the same time,
	every solution $(x^-,y^-)\in\setX\times\Delta^-$ of the augmented game $\gameaug^-$ satisfies the following VI:
	\begin{align}\label{eq:eq_constrScn_m}
		&\sum_{i\in\setN} (x_i -x_i^-)\trasp \nabla_{x_i}\big(f_i(x^-)+\sum_{m=1}^{M-1} y_m^- g(x^-,\theta_m)\big) \nonumber \\
		&+ \sum_{m=1}^{M-1} y_m^-  g(x^-,\theta_m) - \max_{m\in\{1,\ldots,M-1\}}g(x^-,\theta_m) \geq 0.
	\end{align}
Set $x^- = \xopt$ and $y^- = (\yopt_m)_{m=1}^{M-1}$. Under this choice $(x^-,y^-)$ satisfies \eqref{eq:eq_constrScn_m}, as $\max_{m\in\{1,\ldots,M-1\}}g(\xopt,\theta_m) \leq \max_{m\in\{1,\ldots,M\}}g(\xopt,\theta_m)$. Equivalently, $(\xopt,(\yopt_m)_{m=1}^{M-1})$ is an NE for $\gameaug^-$, and by applying Proposition~\ref{prop_eqNE_VIsol} to $\game^-$ and $\gameaug^-$ we have that $\xopt$ is an NE for $\game^-$. 
However, due to the uniqueness assumption, $\xopt$ has to be the only NE of $\game^-$, showing that $\xopt = \Phi(\theta_1,\ldots,\theta_{M-1})$.

Following the same procedure, removing one by one all samples for which the associated elements of $\yopt$ are zero, shows that $\xopt = \Phi(\theta_1,\ldots,\theta_M) = \Phi(\{\theta_m\}_{m\in \setY^{\ast}})$, thus concluding the proof of the first part.

\emph{Part 2: Uniqueness of NE aggregate.} 
The proof follows the same arguments as in Part 1 with the following modifications.
The derivation until the discussion right after \eqref{eq:eq_constrScn_m} remains unaltered, showing that $(\xopt,(\yopt_m)_{m=1}^{M-1})$ is a NE of $\gameaug^-$.  
To prove that $\xopt = \Phi(\theta_1,\ldots,\theta_{M-1})$ it suffices to show that $(\xopt,(\yopt)_{m=1}^{M-1})$ is the \emph{minimum norm} NE of $\gameaug^-$. We thus assume for the sake of contradiction that $(\hat{x},\hat{y})\in\setX\times\Delta^-$ is the NE of $\gameaug^-$ that achieves the minimum norm,  i.e., $ \|(\hat{x},\hat{y})\|^2 < \|(\xopt,(\yopt)_{m=1}^{M-1})\|^2$. 
We distinguish two cases:

\emph{Case 1: $g(\sigma(\hat{x}), \theta_M) \leq \max_{m \in \{1,\ldots,M-1\}}g(\sigma(\hat{x}),\theta_m)$.} 
Under this condition observe that $(\hat{x},(\hat{y}^\traspo,0)\trasp)$ satisfies the VI in \eqref{eq_constrScn3bis} for the game with $M$ samples. 
However, as $(\xopt,\yopt)$ is the minimum norm equilibrium for that game, we have that $\|(\xopt,\yopt)\|^2 \leq \|(\hat{x},(\hat{y}^\traspo,0)\trasp)\|^2$. Overall, recalling $\yopt_M = 0$, $\|(\xopt,(\yopt)_{m=1}^{M-1})\|^2 = \|(\xopt,\yopt)\|^2 \leq \|(\hat{x},(\hat{y}^\traspo,0)\trasp)\|^2 = \|(\hat{x},\hat{y})\|^2$, thus establishing a contradiction. We can then show that $\xopt = \Phi(\theta_1,\ldots,\theta_M) = \Phi(\{\theta_m\}_{m\in \setY^{\ast}})$ as in the last paragraph of Part 1.

\emph{Case 2: $g(\sigma(\hat{x}), \theta_M) > \max_{m \in \{1,\ldots,M-1\}}g(\sigma(\hat{x}),\theta_m)$.}
We will show that, under our assumptions, this case cannot occur. By the uniqueness assumption we have that $\sigma(\hat{x}) = \sigma(\xopt)$ for any equilibrium $\hat{x} \neq \xopt$ (the NE is not necessarily unique, but all equilibria have the same aggregate).
We then have
\begin{align}\label{eq:contr}
g( \sigma(\xopt), \theta_M) &{} = g(\sigma(\hat{x}), \theta_M)\nonumber\\  
  &{} > \max_{m \in \{1,\ldots,M-1\}}g(\sigma(\hat{x}),\theta_m)\nonumber \\
  &{} \geq g(\sigma(\hat{x}),\theta_m) = g(\sigma(\xopt),\theta_m),
\end{align}
for any $m\in{1,\ldots,M-1}$. Since \eqref{eq:contr} holds for any $m$,
\begin{equation}
	g(\sigma(\xopt), \theta_M) > \max_{m \in \{1,\ldots,M-1\}} g(\sigma(\xopt),\theta_m).
\end{equation}
Consider now \eqref{eq_constr_max}. 
By direct computation of the KKT optimality conditions \cite[\S 6.2.1]{RustemHoweBOOK} of \eqref{eq_epigame} and \eqref{eq_pullout_coord}, respectively, it can be verified that the decision variable $y \in \Delta$ introduced in \eqref{eq_pullout_ind}--\eqref{eq_pullout_coord} is a \emph{shadow price} for the constraint \eqref{eq_constr_max}.
Then, by the complementary slackness condition,
\begin{equation}\label{eq_KKTy}
y_m^{\ast}\left(g(\sigma(\xopt),\theta_m)-\gammaopt\right) = 0,\; \forall m\in\{1,\ldots,M\}.
\end{equation}
Since $y_M = 0$ implies $g(\sigma(\xopt),\theta_M) \leq \gammaopt$ we obtain
\begin{align}
\max_{m \in \{1,\ldots,M\}} g(\sigma(\xopt),\theta_m) = \max_{m \in \{1,\ldots,M-1\}} g(\sigma(\xopt),\theta_m). \label{eq:nonactive}
\end{align}
From \eqref{eq:nonactive} it follows $\max_{m \in \{1,\ldots,M-1\}}g(\sigma(\xopt),\theta_m) = \max_{m \in \{1,\ldots,M\}}g(\sigma(\xopt),\theta_m)\geq g(\sigma(\xopt),\theta_M)$, which contradicts \eqref{eq:contr} and concludes the proof.
\end{proof}

Based on the shadow price interpretation of $y$ (see proof of Proposition~\ref{prop_apost_eval}) notice that if $g(\xopt,\theta_m) < \gammaopt$ (inactive constraint) then $\yopt_m = 0$. Note that samples with $\yopt_m=0$ can be removed without altering $\xopt$ due to the imposed uniqueness requirements; otherwise, the feasibility region of the VI in \eqref{eq:eq_constrScn_m} may enlarge, possibly resulting in a different minimum norm NE.
Moreover, it should be noted that Proposition~\ref{prop_apost_eval} does not provide guarantees that a minimal cardinality compression set is determined; this can be obtained by Algorithm~\ref{alg_suppconstenum} (see also \cite{CampiEtAl2018TAC}). However, the important implication of Proposition~\ref{prop_apost_eval} is that the cardinality of a compression set is readily available by inspecting $\yopt$.

%% file: example.tex
\section{Case study: Electric vehicle charging control}
\label{sec_example}

\subsection{Problem set-up}
\label{sec_ex_setup}

We consider a stylized electric vehicle (EV) charging control problem, with EVs being risk-averse, selfish entities interested in minimizing their own cost. Let $\{1,\ldots,N\}$ index the finite population of EV vehicles/agents.
We denote by $x_i \in \mathbb{R}^n$ the demand profile each EV seeks to determine over $n$ time slots, where for simplicity these are taken to be of unit length (1 hour). Vehicles' charging strategy is in response to a pricing signal received from a coordinator, which in turn depends on the demand profiles of all agents. We consider price to be an affine function of the aggregate strategy $\sigma(x)\colon x \mapsto \sum_{i\in\setN}x_i$, but other choices are also supported by our theoretical analysis.
Price is subject to uncertainty, e.g., externalities acting on the energy spot market, encoded by the random variable $\theta\in\Theta$, which we model by means of scenarios. In particular, each scenario is a realization of prices along the considered $n$-slot interval. 
Note that these scenarios are i.i.d., however, each of them is a finite horizon path, whose entries can be correlated. Each agent $i=1,\ldots,N$ aims at minimizing
\begin{align}
&f_i(x_i,\xmenoi) + \max\limits_{m \in\{1,\ldots,M\}} g(x_i,\xmenoi,\theta_m) =  \nonumber \\ & x_i\trasp (A_0 \sigma(x) + b_0)+ \frac{1}{N} \max\limits_{m \in\{1,\ldots,M\}} \sigma(x)\trasp (A_m \sigma(x) + b_m),
\end{align}
where $A_m\in\mathbb{R}^{n\times n}$, for $m = 0,1,\ldots,M$, are diagonal matrices, and $b_m\in\mathbb{R}^{n}$. 
Moreover, we assume the charging operations are subject to
$\setX_i = \{x_i\in\mathbb{R}^n:\; \mathbf{1}\trasp x_i \geq E_i,\; 0 \leq x_{ij} \leq P_i, ~\forall j=1,\ldots,n\}$,
where $E_i,P_i \in \mathbb{R}$ designate the desired final state of charge (SoC) and the maximum power deliverable by the charger, respectively. 

We analyse the results of several randomly generated cases, differing in the parameters characterizing the EV constraints $\setX_i$, selected from a uniform random distribution: specifically, $P_i\in[6,15]$ kW, and $E_i$ is chosen to be feasible in the specified time interval ($\sim$0--35 kWh per 12 h interval).
The pairs $\{A_m,b_m\}_{m = 1}^{M}$ are i.i.d.~extracted from a lognormal distribution for the diagonal entries of $A_m$, and a uniform distribution for the vectors $b_m$. The nominal electricity price, i.e., the diagonal entries $\{a_t\}_{t=1}^n$ of the matrix $A_0$, have been derived by rescaling a winter weekday demand profile in the UK~\cite{dem_data01}, whereas $b_0 = 0$. 

It should be noted that even though this example fits in the class of aggregative games, it does not necessarily meet the uniqueness requirement of the second part of Proposition \ref{prop_apost_eval}. However, we have empirically observed that the main conclusion of the proposition still holds, namely, the minimum norm solution returned by Algorithm~\ref{alg_iterprox} remains unaltered when the algorithm is fed only with the samples with indices in $\setY^{\ast}$. Informally, this happens due to the fact that for any feasible problem instance $\mathbf{1}\trasp x_i \geq E_i$ will always be binding at the optimum, and as result $\mathbf{1}\trasp \sigma(x)$ will be constant for any NE $x$ (see also transparent plane in Figure \ref{fig_suppconstr}); a detailed investigation of this issue is topic of current research.

\begin{figure}
	\centering
	\includegraphics[width=8.4cm, trim={0 0 0 0.55cm},clip]{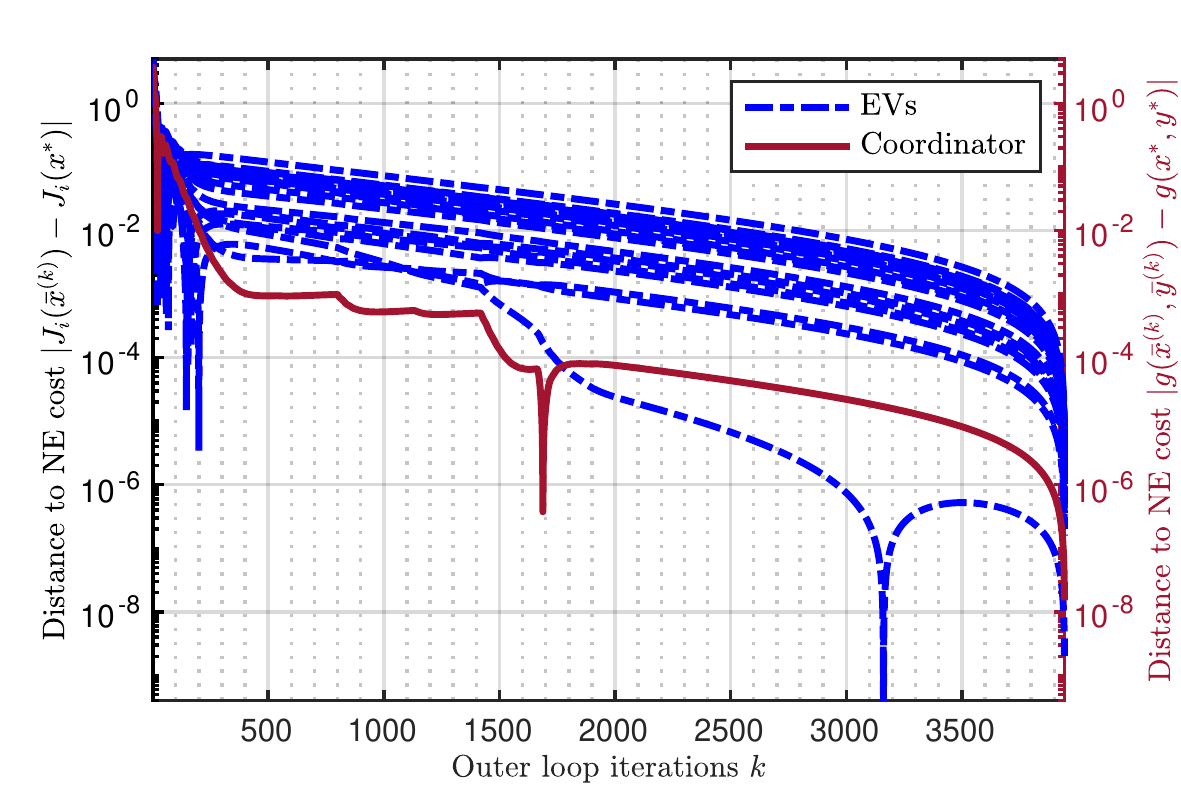}    
	\vspace{-8pt}
	\caption{Convergence of EV agents' and coordinator's costs (outer loop), for $N = 20$, $n=24$, $M=500$: a dominant linear convergence rate can be observed. $J_i$ denotes the cost of each player $i = 1,\ldots,N$, and the coordinator of the augmented game $\gameaug$, defined by the objective functions in~\eqref{eq_pullout_ind} and \eqref{eq_pullout_coord}, respectively.} 
	\label{fig_costconv}
\end{figure}
\begin{figure}
	\centering
	\includegraphics[width=8.0cm, trim={0 0 0 0.5cm},clip]{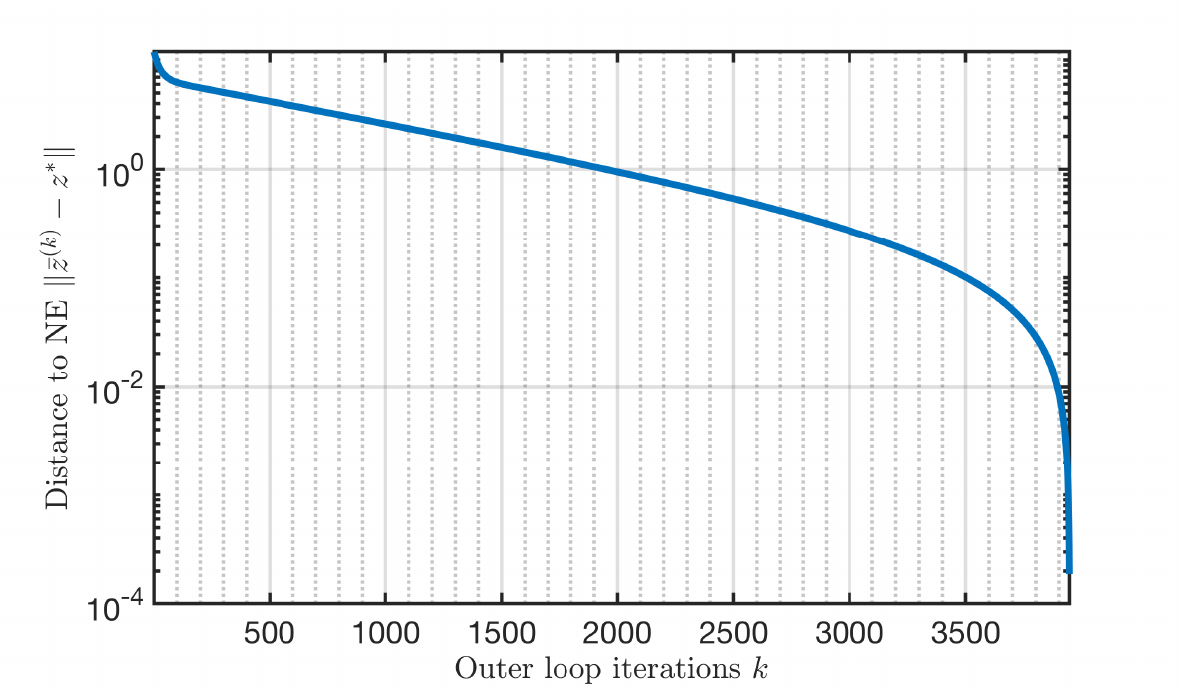}    
	\vspace{-8pt}
	\caption{Convergence of the solution vector to the NE (outer loop) for $N = 20$, $n=24$, $M=500$: a dominant linear convergence rate can be observed.} 
	\label{fig_solconv}
\end{figure}

\begin{table}
	\centering
	\caption{Empirical validation of the \emph{a posteriori} result of Theorem~\ref{thm_apost}.}\label{tab_violrate}
	\vspace{-2pt}
	\begin{tabular}{rcccc}\hline
		$d^*$ [\hspace{0.5ex}-\hspace{0.5ex}] & 4 & 6 & 7 & 9 \\ 
		Empirical~$V(\xopt)$ [\%] & 0.98 & 1.09 & 1.26 & 1.33  \\ 
		$\varepsilon(d^*)$ - Thm~\ref{thm_apost} [\%] & 8.06 & 9.76 & 10.55 & 12.06  \\
		$\varepsilon(d^*)$ - bound \eqref{eq:eps_tight} [\%] & 5.30 & 6.11 & 6.49 & 7.22  \\\hline 
	\end{tabular}
\end{table}
\subsection{Simulation results}
\label{sec_ex_sim}
The NE charging schedules have been obtained by implementing Algorithm~\ref{alg_iterprox} with $\gamma_{\mathrm{inn}} = 10^{-14}$, $\gamma_{\mathrm{out}} = 10^{-5}$ and $\tau \in [4,6]$. At each iteration of the proposed algorithm a quadratic optimization needs to be solved; this was performed on a dual-core 7th gen.~Intel processor using MATLAB. Figures ~\ref{fig_costconv}--\ref{fig_solconv} show the convergence of Algorithm~\ref{alg_iterprox} in the computation of the NE for $N=20$ EVs, $n = 24$ and $M = 500$. A dominant linear convergence rate can be observed, and less than 4000 outer loop iterations were needed to meet the desired exit accuracy $\gamma_{\mathrm{out}}$. The inner loop  enjoys similar convergence rate (not shown for space reasons), and less than 30 iterations (15 in average) are needed to achieve an error smaller than $\gamma_{\mathrm{inn}}$.

To validate the \emph{a posteriori} result of Theorem~\ref{thm_apost}, Table~\ref{tab_violrate} shows the average robustness performance of several solutions (with $N=20$, $n = 24$) obtained from different sets of $M=500$ samples, grouped according to the \emph{a posteriori} observed compression cardinality $d^*$; we have set $\beta = 10^{-6}$. The violation rate $V(\xopt)$ of each solution is empirically computed using $10^6$ newly extracted samples (according to the same aforementioned distributions) and counting the fraction of them that result in a change of the computed NE.
Consistently with \cite{CampiEtAl2018TAC}, we note that the observed value of $d^*$ is indicative of the confidence level on the equilibrium robustness. 
The experimental results are compared with the theoretical bound provided by Theorem~\ref{thm_apost} (third row). For non-degenerate problems, the conservatism of the latter can be reduced by employing the tighter expression for $\varepsilon(\cdot)$ reported in \eqref{eq:eps_tight}, leading to the fourth row of Table~\ref{tab_violrate}. However, note that in general it is difficult to verify whether a given problem is non-degenerate, thus preventing the use of \eqref{eq:eps_tight}.
We observe that a bound of $\sim$2--3\% could have been achieved with Theorem \ref{thm_apost} by increasing the sample size to $M=2000$.

\begin{figure}
	\centering
	\includegraphics[width=8.4cm, trim={0 0 0 1.1cm},clip]{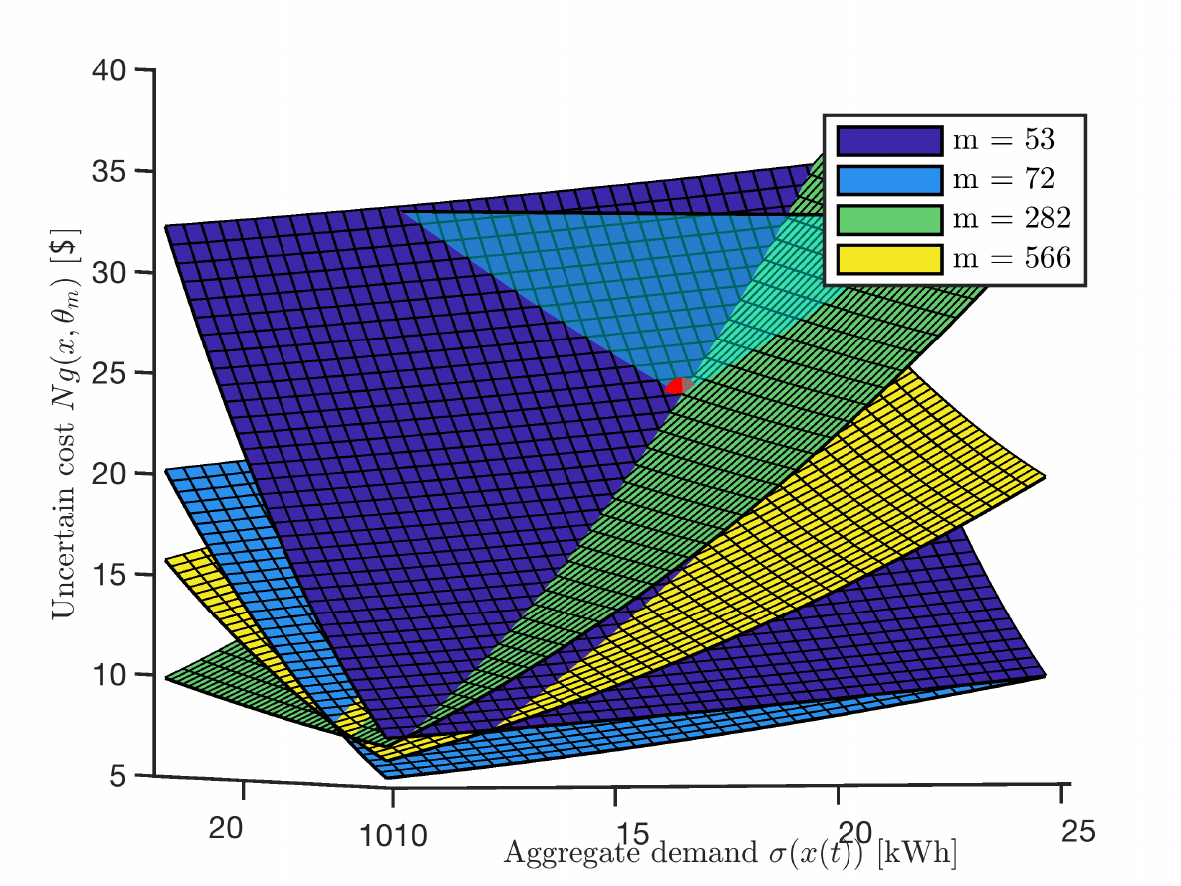}    
	\vspace{-9pt}
	\caption{Case with $N=20$, $n=2$: Uncertain cost component $Ng(\xopt,\theta_m)$ for a subset $m\in\{53,72,282,566\}$ of the $M = 1000$ samples used for the derivation of the NE $\xopt$. In this case $d^*=2$, with $\{\theta_{53},\theta_{282}\}$ supporting the solution together with the SoC constraint which is binding in this case. This is the transparent plane, representing all aggregate strategies fulfilling $\sigma(x)|_{t=1}+\sigma(x)|_{t=2} = \sum_{i\in\setN} E_i$ over the considered time interval. $\sigma(\xopt)$, visible in red, lies at the intersection of the three surfaces.} 
	\label{fig_suppconstr}
\end{figure}

\begin{figure}
	\centering
	\includegraphics[width=8.4cm, trim={0 0 0 0.2cm},clip]{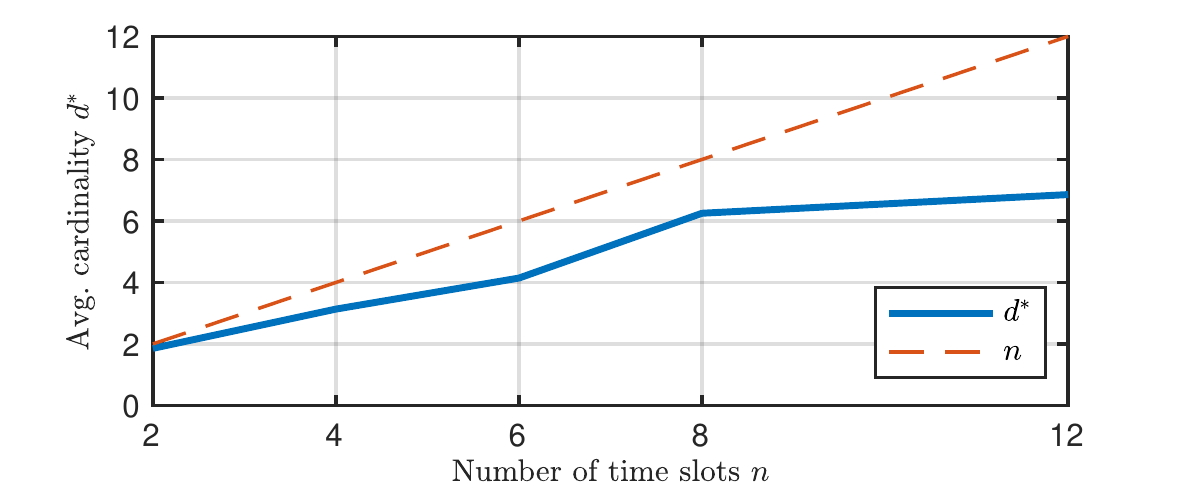}    
	\vspace{-9pt}
	\caption{Empirical validation of the \emph{a priori} result of Theorem~\ref{thm_apriori}. The plot shows the average compression set cardinality $d^*$ observed over 50 trials, corresponding to different randomly generated cases corresponding to different values of $n$ and $M$. In all cases, $d^*$ is bounded by $n$, and hence also by the theoretical bound $(n+1)N$.} 
	\label{fig_T_vs_barn}
\end{figure}

A visual representation of the concept of compression set is given in Fig.~\ref{fig_suppconstr}. The plot depicts the curves expressing the uncertain cost term $Ng(\xopt,\theta_m)$ associated to a subset $m\in\{53,72,282,566\}$ of the $M = 1000$ samples used for the derivation of the NE $\xopt$. Values are plotted as a function of the aggregate demand $(\sigma(x)|_{t=1},\sigma(x)|_{t=2})$ on an interval around $\sigma(\xopt)$. In this case the (minimal) compression cardinality is $d^*=2$, with $\{\theta_{53},\theta_{282}\}$ supporting the solution together with the constraint on the target SoC which is binding in this case (transparent plane). Note that in this instance the constraints on the power rate $P_i$ are not active, and omitted from the plot for clarity.

We now investigate numerically the validity of the \emph{a priori} result of Theorem \ref{thm_apriori}.
Fig.~\ref{fig_T_vs_barn} shows the average compression set cardinality $d^*$ (solid line) observed over 50 trials, corresponding to different randomly generated cases corresponding to different values of $n$ and $M$.
In all cases, $d^*$ is bounded by $(n+1)N$ as suggested by Theorem \ref{thm_apriori}. In fact, the empirically calculated cardinality $d^*$ is significantly lower, suggesting that in this case study an \emph{a posteriori} quantification is less conservative. Moreover, in all our numerical investigations we noticed that $d^*\leq n$, i.e., the empirical estimate of the compression set cardinality is independent of the number of agents and is bounded by the number of individual decision variables (dashed line). We conjecture that for the aggregative EV charging game considered here, involving affine price functions, the so called support rank (see \cite{SchildbachEtAl2013SIAM} for a definition) offers a tighter bound on the compression set cardinality compared to the total number of decision variables $(n+1)N$.

 \section{Concluding remarks} \label{sec_conclusion}
We considered the problem of NE computation in multi-agent games in the presence of uncertainty, and accompanied them with \emph{a priori} and \emph{a posteriori} certificates regarding the probability that the NE equilibrium remains unchanged when a new uncertainty realization is encountered.

Current work is concentrated towards relaxing the uniqueness requirements underpinning the compression set quantification of Section~\ref{sec_apost_eval} for the class of aggregative games.

%% file: TAC_FM_2019.bbl
\begin{thebibliography}{10}
\providecommand{\url}[1]{#1}
\csname url@rmstyle\endcsname
\providecommand{\newblock}{\relax}
\providecommand{\bibinfo}[2]{#2}
\providecommand\BIBentrySTDinterwordspacing{\spaceskip=0pt\relax}
\providecommand\BIBentryALTinterwordstretchfactor{4}
\providecommand\BIBentryALTinterwordspacing{\spaceskip=\fontdimen2\font plus
\BIBentryALTinterwordstretchfactor\fontdimen3\font minus
  \fontdimen4\font\relax}
\providecommand\BIBforeignlanguage[2]{{%
\expandafter\ifx\csname l@#1\endcsname\relax
\typeout{** WARNING: IEEEtran.bst: No hyphenation pattern has been}%
\typeout{** loaded for the language `#1'. Using the pattern for}%
\typeout{** the default language instead.}%
\else
\language=\csname l@#1\endcsname
\fi
#2}}

\bibitem{LAMNABHILAGARRIGUEEtAl2017CPS}
F.~Lamnabhi-Lagarrigue, A.~Annaswamy, S.~Engell, A.~Isaksson, P.~Khargonekar,
  R.~M. Murray, H.~Nijmeijer, T.~Samad, D.~Tilbury, and P.~V. den Hof,
  ``Systems \& control for the future of humanity, research agenda: Current and
  future roles, impact and grand challenges,'' \emph{Annual Reviews in
  Control}, vol.~43, pp. 1 -- 64, 2017.

\bibitem{Tony2019}
H.~{Le Cadre}, I.~Mezghani, and A.~Papavasiliou, ``A game-theoretic analysis of
  transmission-distribution system operator coordination,'' \emph{European
  Journal of Operational Research}, vol. 274, no.~1, pp. 317 -- 339, 2019.

\bibitem{DePaolaEtAl2018CDC}
A.~{De Paola}, F.~{Fele}, D.~{Angeli}, and G.~{Strbac}, ``Distributed
  coordination of price-responsive electric loads: A receding horizon
  approach,'' in \emph{2018 IEEE Conference on Decision and Control (CDC)}, Dec
  2018, pp. 6033--6040.

\bibitem{AtzeniEtAl2014TSP}
I.~{Atzeni}, L.~G. {Ordóñez}, G.~{Scutari}, D.~P. {Palomar}, and J.~R.
  {Fonollosa}, ``Noncooperative day-ahead bidding strategies for demand-side
  expected cost minimization with real-time adjustments: A {GNEP} approach,''
  \emph{IEEE Transactions on Signal Processing}, vol.~62, no.~9, pp.
  2397--2412, May 2014.

\bibitem{Oren2004}
R.~Kamat and S.~Oren, ``Two-settlement systems for electricity markets under
  network uncertainty and market,'' \emph{Power. Journal of Regulatory
  Economics}, vol.~25, pp. 5--37, Jan 2004.

\bibitem{Oren2008}
J.~Yao, I.~Adler, and S.~Oren, ``Modelling and computation of two-settlement
  oligopolistic equilibrium in a congested electricity network,''
  \emph{Operations Research}, vol.~56, no.~1, pp. 34--47, Jan 2008.

\bibitem{ScutariEtAl2014}
G.~Scutari, F.~Facchinei, J.~S. Pang, and D.~P. Palomar, ``Real and complex
  monotone communication games,'' \emph{IEEE Transactions on Information
  Theory}, vol.~60, no.~7, pp. 4197--4231, July 2014.

\bibitem{Krawczyk2007}
J.~Krawczyk, ``Numerical solutions to coupled-constraint (or generalised
  {N}ash) equilibrium problems,'' \emph{Computational Management Science},
  vol.~4, pp. 183--204, Nov 2007.

\bibitem{Ohlin2012}
J.~Ohlin, ``Nash equilibrium and international law,'' \emph{European Journal of
  International Law}, vol.~23, no.~4, pp. 915--940, Dec 2012.

\bibitem{Nash1951}
J.~Nash, ``Non-cooperative games,'' \emph{Annals of Mathematics}, vol.~54,
  no.~2, pp. 286--295, 1951.

\bibitem{BasarOlsder1998BOOK}
T.~Başar and G.~Olsder, \emph{Dynamic Noncooperative Game Theory, 2nd
  Edition}.\hskip 1em plus 0.5em minus 0.4em\relax Society for Industrial and
  Applied Mathematics, 1998.

\bibitem{FacchineiKanzow2007}
F.~Facchinei and C.~Kanzow, ``Generalized {N}ash equilibrium problems,''
  \emph{4OR}, vol.~5, no.~3, pp. 173--210, Sep 2007.

\bibitem{Grammatico2017TAC}
S.~{Grammatico}, ``Dynamic control of agents playing aggregative games with
  coupling constraints,'' \emph{IEEE Transactions on Automatic Control},
  vol.~62, no.~9, pp. 4537--4548, Sep. 2017.

\bibitem{DeoriEtAl2018AUT}
L.~Deori, K.~Margellos, and M.~Prandini, ``Price of anarchy in electric vehicle
  charging control games: When {N}ash equilibria achieve social welfare,''
  \emph{Automatica}, vol.~96, pp. 150 -- 158, 2018.

\bibitem{PaccagnanEtAl2018TAC}
D.~{Paccagnan}, B.~{Gentile}, F.~{Parise}, M.~{Kamgarpour}, and J.~{Lygeros},
  ``Nash and {W}ardrop equilibria in aggregative games with coupling
  constraints,'' \emph{IEEE Transactions on Automatic Control}, pp. 1--1, 2018.

\bibitem{PaccagnanEtAl2018CSL}
D.~{Paccagnan}, F.~{Parise}, and J.~{Lygeros}, ``On the efficiency of {N}ash
  equilibria in aggregative charging games,'' \emph{IEEE Control Systems
  Letters}, vol.~2, no.~4, pp. 629--634, Oct 2018.

\bibitem{FeleEtAl2018ARC}
F.~Fele, A.~{De Paola}, D.~Angeli, and G.~Strbac, ``A framework for
  receding-horizon control in infinite-horizon aggregative games,''
  \emph{Annual Reviews in Control}, vol.~45, pp. 191 -- 204, 2018.

\bibitem{Couchman20051283}
P.~Couchman, B.~Kouvaritakis, M.~Cannon, and F.~Prashad, ``Gaming strategy for
  electric power with random demand,'' \emph{IEEE Transactions on Power
  Systems}, vol.~20, no.~3, pp. 1283--1292, 2005.

\bibitem{SINGH2016640}
V.~V. Singh, O.~Jouini, and A.~Lisser, ``Existence of {N}ash equilibrium for
  chance-constrained games,'' \emph{Operations Research Letters}, vol.~44,
  no.~5, pp. 640 -- 644, 2016.

\bibitem{PangEtAl2017}
J.-S. Pang, S.~Sen, and U.~V. Shanbhag, ``Two-stage non-cooperative games with
  risk-averse players,'' \emph{Mathematical Programming}, vol. 165, no.~1, pp.
  235--290, Sep 2017.

\bibitem{Ravat20111168}
U.~Ravat and U.~Shanbhag, ``On the characterization of solution sets of smooth
  and nonsmooth convex stochastic {N}ash games,'' \emph{SIAM Journal on
  Optimization}, vol.~21, no.~3, pp. 1168--1199, 2011.

\bibitem{KoshalEtAl2013TAC}
J.~{Koshal}, A.~{Nedic}, and U.~V. {Shanbhag}, ``Regularized iterative
  stochastic approximation methods for stochastic variationalinequality
  problems,'' \emph{IEEE Transactions on Automatic Control}, vol.~58, no.~3,
  pp. 594--609, March 2013.

\bibitem{XuZhang2013}
H.~Xu and D.~Zhang, ``Stochastic {N}ash equilibrium problems: sample average
  approximation and applications,'' \emph{Computational Optimization and
  Applications}, vol.~55, no.~3, pp. 597--645, Jul 2013.

\bibitem{YuEtAl2017Learning}
C.~{Yu}, M.~{van der Schaar}, and A.~H. {Sayed}, ``Distributed learning for
  stochastic generalized {N}ash equilibrium problems,'' \emph{IEEE Transactions
  on Signal Processing}, vol.~65, no.~15, pp. 3893--3908, Aug 2017.

\bibitem{LeiShanbhag2017CDC}
J.~{Lei} and U.~V. {Shanbhag}, ``A randomized inexact proximal best-response
  scheme for potential stochastic {N}ash games,'' in \emph{2017 IEEE 56th
  Annual Conference on Decision and Control (CDC)}, Dec 2017, pp. 1646--1651.

\bibitem{AghassiBertsimas2006}
M.~Aghassi and D.~Bertsimas, ``Robust game theory,'' \emph{Mathematical
  Programming}, vol. 107, no.~1, pp. 231--273, Jun 2006.

\bibitem{Nishimura09robustnash}
R.~Nishimura, S.~Hayashi, and M.~Fukushima, ``Robust {N}ash equilibria in
  $n$-person non-cooperative games: Uniqueness and reformulation,''
  \emph{Pacific J. Optim.}, no.~5, 2009.

\bibitem{ZazoEtAl2017}
J.~Zazo, S.~Zazo, and S.~V. Macua, ``Robust worst-case analysis of demand-side
  management in smart grids,'' \emph{IEEE Transactions on Smart Grid}, vol.~8,
  no.~2, pp. 662--673, March 2017.

\bibitem{HuFukushima2013}
M.~Hu and M.~Fukushima, ``Existence, uniqueness, and computation of robust
  {N}ash equilibria in a class of multi-leader-follower games,'' \emph{SIAM
  Journal on Optimization}, vol.~23, no.~2, pp. 894--916, 2013.

\bibitem{ConejoEtAl2010BOOK}
A.~J. Conejo, M.~Carri{\'o}n, and J.~M. Morales, \emph{Uncertainty
  Characterization via Scenarios}.\hskip 1em plus 0.5em minus 0.4em\relax
  Boston, MA: Springer US, 2010, pp. 63--119.

\bibitem{FloydWarmuth1995}
S.~Floyd and M.~Warmuth, ``Sample compression, learnability, and the
  {V}apnik-{C}hervonenkis dimension,'' \emph{Machine Learning}, vol.~21, no.~3,
  pp. 269--304, Dec 1995.

\bibitem{AlamoEtAl2009TAC}
T.~{Alamo}, R.~{Tempo}, and E.~F. {Camacho}, ``Randomized strategies for
  probabilistic solutions of uncertain feasibility and optimization problems,''
  \emph{IEEE Transactions on Automatic Control}, vol.~54, no.~11, pp.
  2545--2559, Nov 2009.

\bibitem{MargellosEtAl2015TAC}
K.~{Margellos}, M.~{Prandini}, and J.~{Lygeros}, ``On the connection between
  compression learning and scenario based single-stage and cascading
  optimization problems,'' \emph{IEEE Transactions on Automatic Control},
  vol.~60, no.~10, pp. 2716--2721, Oct 2015.

\bibitem{CalafioreCampi2006TAC}
G.~C. Calafiore and M.~C. Campi, ``The scenario approach to robust control
  design,'' \emph{IEEE Transactions on Automatic Control}, vol.~51, no.~5, pp.
  742--753, May 2006.

\bibitem{CampiEtAl2018TAC}
M.~C. {Campi}, S.~{Garatti}, and F.~A. {Ramponi}, ``A general scenario theory
  for nonconvex optimization and decision making,'' \emph{IEEE Transactions on
  Automatic Control}, vol.~63, no.~12, pp. 4067--4078, Dec 2018.

\bibitem{CampiGaratti2008SIAM}
M.~Campi and S.~Garatti, ``The exact feasibility of randomized solutions of
  uncertain convex programs,'' \emph{SIAM Journal on Optimization}, vol.~19,
  no.~3, pp. 1211--1230, 2008.

\bibitem{FeleMargellos2019}
F.~{Fele} and K.~{Margellos}, ``Probabilistic sensitivity of {N}ash equilibria
  in multi-agent games: a wait-and-judge approach,'' in \emph{2019 IEEE 58th
  Conference on Decision and Control (CDC)}, 2019, pp. 5026--5031.

\bibitem{PaccagnanCampi2019}
D.~{Paccagnan} and M.~C. {Campi}, ``The scenario approach meets uncertain game
  theory and variational inequalities,'' in \emph{2019 IEEE 58th Conference on
  Decision and Control (CDC)}, 2019, pp. 6124--6129.

\bibitem{FacchineiEtAl2014}
F.~Facchinei, J.-S. Pang, and G.~Scutari, ``Non-cooperative games with minmax
  objectives,'' \emph{Computational Optimization and Applications}, vol.~59,
  no.~1, pp. 85--112, Oct 2014.

\bibitem{Noor1996}
M.~A. Noor, ``Quasi variational inequalities,'' \emph{Applied Mathematics
  Letters}, vol.~1, no.~4, pp. 367 -- 370, 1988.

\bibitem{Kannan_etal_2013}
A.~Kannan, U.~Shanbhag, and H.~Kim, ``Addressing supply-side risk in uncertain
  power markets: stochastic {N}ash models, scalable algorithms and error
  analysis,'' \emph{Optimization Methods and Software}, vol.~28, no.~5, pp.
  1095--1138, 2013.

\bibitem{CampiGaratti2018MATH}
M.~C. Campi and S.~Garatti, ``Wait-and-judge scenario optimization,''
  \emph{Mathematical Programming}, vol. 167, no.~1, pp. 155--189, Jan 2018.

\bibitem{Calafiore2010SIAM}
G.~Calafiore, ``Random convex programs,'' \emph{SIAM Journal on Optimization},
  vol.~20, no.~6, pp. 3427--3464, 2010.

\bibitem{FacchineiPang2009BOOK}
F.~Facchinei and J.-S. Pang, \emph{Finite-Dimensional Variational Inequalities
  and Complementarity Problems}.\hskip 1em plus 0.5em minus 0.4em\relax
  Springer-Verlag New York, 2003.

\bibitem{RustemHoweBOOK}
B.~Rustem and M.~Howe, \emph{Algorithms for Worst-Case Design and Applications
  to Risk Management}.\hskip 1em plus 0.5em minus 0.4em\relax Princeton
  University Press, 2002.

\bibitem{KannanShanbhag2012SIAM}
A.~Kannan and U.~Shanbhag, ``Distributed computation of equilibria in monotone
  {N}ash games via iterative regularization techniques,'' \emph{SIAM Journal on
  Optimization}, vol.~22, no.~4, pp. 1177--1205, 2012.

\bibitem{YiPavel2019}
P.~{Yi} and L.~{Pavel}, ``Distributed generalized {N}ash equilibria computation
  of monotone games via double-layer preconditioned proximal-point
  algorithms,'' \emph{IEEE Transactions on Control of Network Systems}, vol.~6,
  no.~1, pp. 299--311, 2019.

\bibitem{BelgioiosoGrammatico2020}
G.~Belgioioso and S.~Grammatico, ``Semi-decentralized generalized {N}ash
  equilibrium seeking in monotone aggregative games,'' 2020.

\bibitem{dem_data01}
\BIBentryALTinterwordspacing
{National Grid}. (2019, Mar.) Historical demand data. [Online]. Available:
  \url{https://www.nationalgrideso.com/balancing-data/data-explorer}
\BIBentrySTDinterwordspacing

\bibitem{SchildbachEtAl2013SIAM}
G.~Schildbach, L.~Fagiano, and M.~Morari, ``Randomized solutions to convex
  programs with multiple chance constraints,'' \emph{SIAM Journal on
  Optimization}, vol.~23, no.~4, pp. 2479--2501, 2013.

\end{thebibliography}
